\documentclass[11pt, a4paper]{amsart}

\usepackage[english]{babel}
\usepackage{amsmath,enumerate, amsfonts, amssymb,amsthm, xypic}
\usepackage[dvips]{graphics} 

\newtheorem{thm}{Theorem}[section]
\newtheorem{lemma}[thm]{Lemma}
\newtheorem{prop}[thm]{Proposition}

\newenvironment{demo}{\noindent{\it Proof.}\,}{\begin{flushright}
  \,$\Box$ \smallskip \end{flushright}}
\theoremstyle{definition}
\newtheorem{defi}[thm]{Definition}

\newtheorem{rmk}[thm]{Remark}

\DeclareMathOperator{\R}{\mathbf R}

\DeclareMathOperator{\N}{\mathbf N}

\DeclareMathOperator{\PP}{\mathbf P}

\DeclareMathOperator{\id}{id}
\DeclareMathOperator{\codim}{codim}
\DeclareMathOperator{\Int}{Int}

\DeclareMathOperator{\Sing}{Sing}
\DeclareMathOperator{\Ima}{Im}
\DeclareMathOperator{\const}{const}
\DeclareMathOperator{\Reg}{Reg}

\DeclareMathOperator{\Ker}{Ker}

\title{On almost Blow-analytic equivalence}
\author{Goulwen Fichou and Masahiro Shiota}
\thanks{The first author has been supported by the ANR project ANR-08-JCJC-0118-01.}
\address{IRMAR (UMR 6625), Universit\'e de Rennes 1, Campus de Beaulieu, 35042 Rennes Cedex, France and
Graduate School of Mathematics, Nagoya University, Chikusa, Nagoya, 
464-8602, Japan}


\begin{document}

\begin{abstract} We study the analytic equivalence of real analytic
 function germs after desingularization and state the cardinality of the classes
 under this equivalence relation. We consider also the Nash case, and
 compare these equivalences with the
 blow-analytic (respectively blow-Nash) equivalence. We prove an
 approximation result after desingularization: Nash function germs that are analytically equivalent after analytic desingularizations are Nash equivalent after Nash desingularizations.
\end{abstract}
\maketitle
In the study of real analytic function singularities, the choice of a relevant equivalence relation is a crucial but difficult topic. After Hironaka Desingularization Theorem \cite{Hi}, which enables to produce functions with only normal crossing singularities after a finite sequence of blowings-up along smooth analytic centers, it seems natural to expect that equivalent real analytic functions should admit similar resolutions of their singularities. In that spirit, we propose in this paper to study the equivalence relation obtained by requiring that two real analytic functions, defined on a compact real analytic manifold, are equivalent if there exist two Hironaka desingularizations such that the modified functions with only normal crossing singularities, obtained after the desingularizations, become analytically equivalent.

This definition is weaker than that of blow-analytic equivalence
introduced by T.~C.~Kuo \cite{Ku}, which seems to be, up to now, the
best candidate to be the real counterpart of the topological
equivalence for complex analytic functions (cf \cite{FP} for a recent
survey). Let $f,g:(\mathbb R^n,0)\longrightarrow (\mathbb R,0)$ be
analytic function germs. They are blow-analytically equivalent in the
sense of \cite{Ku} if
there exist real modifications $\beta_f:M_f \longrightarrow \mathbb R^n$
and $\beta_g:M_g \longrightarrow \mathbb R^n$ and an analytic
isomorphism $\Phi :(M_f, \beta_f^{-1}(0)) \longrightarrow (M_g,
\beta_g^{-1}(0))$ which induces a homeomorphism $\phi:(\mathbb R^n,0)
\longrightarrow  (\mathbb R^n,0)$ such that $f=g \circ \phi$.
In particular, the desingularizations are replaced by the
notion of real modifications (in order to obtain an equivalence
relation) and the analytic isomorphism between the modifications
induces a homeomorphism.

The equivalence relation we propose to study in this paper is called
almost Blow-analytic equivalence in the language of blow-analytic
theory (cf \cite{KW} and \cite{FKK} for the terminology). 
Almost Blow-analytic equivalence is really a different
relation: we exhibit in section \ref{sect2} two functions that are
almost Blow-analytically equivalent but not blow-analytically
equivalent. 
Note that we do not know in general if these relations are effectively
equivalence relations \cite{FKK} (in particular we use the generated equivalence
relations). However we give a proof in proposition \ref{eqrel} of the
known fact that this is
indeed the case if we allow non smooth blowings-up in the
definitions. We focus also on the questions of the cardinality of the
equivalence classes. We prove in particular that the
cardinality of the set of equivalence classes for almost Blow-analytic
equivalence is countable. Therefore it is reasonable to hope
for a classification! 
The proof of that result is based on the study
of the cardinality for the analytic equivalence classes of normal
crossing functions made in \cite{FS}, where we reduced the problem to
the analogous study for Nash functions, i.e. real analytic functions
whose graph is a semi-algebraic set (described by a set of polynomial
equalities and inequalities with real coefficient
polynomials). Actually, the study of almost Blow-Nash equivalence for
Nash functions presents a great interest by itself (cf. \cite{Ko,Fi}
for notions on blow-Nash equivalence). Notably the cardinality of the
set of equivalence classes remains countable even if the underlying Nash manifold is no longer compact.

We prove in theorem \ref{main} that almost Blow-analytically
equivalent Nash function germs are almost Blow-Nash equivalent. This
result can be view as a Nash Approximation Theorem
\cite{CRS} after blowings-up. Note
that this question remains open for blow-analytic equivalence.

In order to prove theorem \ref{main}, we focus on the Nash
approximation of a Hironaka desingularization of a Nash
function. Namely, let $f$ be a Nash function on a Nash manifold $M$
and $X$ a compact semialgebraic subset of $M$. Then
$f$ is in particular a real analytic function on a real analytic
manifold, and by Hironaka Desingularization Theorem \cite{Hi} there
exists a composition $\pi$ of blowings-up along smooth analytic
centers such that $f\circ \pi$ has only normal crossing
singularities on a neighborhood of $\pi^{-1}(X)$. As a main result of the paper, we prove in theorem \ref{Nresol} that each blowing-up
along a smooth analytic center can be approximated by a blowing-up
along a smooth Nash center in such a way that the normal crossing
property of the modified function continues to hold.
In order to approximate a sequence of blowings-up along smooth centers, we focus in section \ref{sect-eu} on a Euclidean realization of such a sequence to describe precisely its behavior under a perturbation of the defining ideal of the centers (cf. lemma \ref{pert}). Combined with N\'eron Desingularization \cite{Sp}, this implies theorem \ref{Nresol}. But this is not sufficient to prove theorem \ref{main} since we need to approximate also the analytic diffeomorphism of the equivalence after the desingularization. To this aim, we need to generalize the Nash Approximation Theorem in \cite{CRS} to a more general noncompact situation (cf. proposition \ref{prop-app}). We obtain as a corollary that analytically equivalent Nash function germs on a compact semialgebraic set in a Nash manifold are Nash equivalent (cf. theorem \ref{thmN}). The last section is devoted to the proof of theorem \ref{main}.

In this paper a manifold means a manifold without boundary, analytic manifolds and maps mean real analytic ones unless otherwise specified, and id stands for the identity map. 



\section{Almost Blow-analytic equivalence}\label{bae}


\subsection{Definition of almost Blow-analytic equivalence}
The classical right equivalence between analytic function germs says that real
analytic function germs $f,g:M\longrightarrow \R$ on an analytic
manifold $M$ are equivalent if there exists an analytic diffeomorphism
$h:M\longrightarrow M$ such that $f=g \circ h$. We are interested in
this paper in weaker notions of equivalence between function germs.

\begin{defi}\label{def1} Let $M$ be an analytic manifold and $f,g:M\longrightarrow \R$ analytic functions on $M$. 
Then $f$ and $g$ are said to be {\it almost Blow-analytically equivalent} if there exist two 
compositions of finite sequences of blowings-up along smooth analytic centers $\pi_f:N\longrightarrow 
M$ and $\pi_g:L\longrightarrow M$ and an analytic diffeomorphism $h:N\longrightarrow L$ so that 
$f\circ\pi_f=g\circ\pi_g\circ h$. 
In the case where there exist such compositions of finite sequences of
blowings-up along smooth analytic centers $\pi_f:N\to M$ and $\pi_g:L\to M$, and analytic 
diffeomorphisms $h:N\to L$ and $\tau:\R\to\R$ such that $\tau\circ f\circ\pi_f=g\circ\pi_g\circ h$, 
$f$ and $g$ are called {\it almost Blow-analytically R-L (=right-left) equivalent}. 

\end{defi}
In this paper, we will be interested also in {\it almost Blow-analytic
 (R-L) equivalence} for germs of analytic functions, whose definition
is similar to definition \ref{def1}.

Here and from now on, we impose some restrictions to the blowings-up,
that are natural thanks to Hironaka Desingularization Theorem. In
particular, we treat only the case where the images of the centers of the 
blowings-up of $\pi_f$ and $\pi_g$ are contained in their singular point sets $\Sing f$ and $\Sing g$ 
respectively, and the center $C$ of each blowing-up is of co-dimension
strictly bigger than one. We assume also that $C$ has only normal crossing with the 
union $D$ of the inverse images of the previous centers, i.e. there exists an analytic local coordinate 
system $(x_1,...,x_n)$ at each point of $C$ such that $C=\{x_1=\cdots=x_k\}$ and $D=\{x_{i_1}\cdots x_{i
_l}=0\}$ for some $1<k\in\N$ and $1\le i_1<\cdots<i_l\le n\in\N$, where $\N=\{0,1,...\}$.

We recall that a {\it semi-algebraic} set is a subset of a Euclidean space which is described by finitely many equalities 
and inequalities of polynomial functions. 
A {\it Nash manifold} is a $C^\omega$ submanifold of a Euclidean space which is semi-algebraic. 
A {\it Nash function} on a Nash manifold is a $C^\omega$ function with semi-algebraic graph. 
A {\it Nash subset} is the zero set of a Nash function on a Nash
manifold. The story of Nash manifolds and Nash maps begins with the
fundamental paper \cite{N} of J. Nash who realized any compact smooth
manifold as a union of some connected components of a real algebraic set. We refer to
\cite{Shiota} for an overview on Nash functions on a Nash subset.

\begin{defi}
Let $M$ be a Nash manifold, $X\subset M$ a semi-algebraic subset and $f,g$ Nash function germs on 
$X$ in $M$. 
Then $f$ and $g$ are said to be {\it almost Blow-Nash equivalent} if there exist open semi-algebraic 
neighborhoods $U$ and $V$ of $X$ in $M$, two compositions of finite sequences of blowings-up along smooth Nash 
centers $\pi_f:N \longrightarrow U$ and $\pi_g:L\longrightarrow V$ and a Nash diffeomorphism $h$ from an open semi-algebraic neighborhood of $\pi_f^{-1}(X)$ in $N$ to one of $\pi_g^{-1}(X)$ in $L$ so that $f$ and $g$ are supposed to be defined on $U$ and $V$, respectively, $f\circ\pi_f=g\circ\pi
_g\circ h$ and $h(\pi^{-1}_f(X))=\pi^{-1}_g(X)$. 
\end{defi}

We define {\it almost Blow-Nash R-L equivalence} similarly to
definition \ref{def1}.

\begin{rmk} We are interested in this paper in global versions of the
  classical blow-analytic equivalence defined by T.-C. Kuo, which
  considers only germs of functions at a point. In the
  analytic case, we will consider the local case and also the
  compact case since the
  desingularization theorem of H. Hironaka \cite{Hi} is valid in
  these situations. In the Nash category, we can deal
  with even the non-isolated situation by \cite{Hi,BM}.
\end{rmk}

The {\it Blow-analytic equivalence} is defined by requiring above $h$ to induce a homeomorphism of $M$. 
We do not know whether the almost Blow-analytic (R-L) equivalence and the Blow-analytic 
(R-L) equivalence give equivalence relations (cf. \cite{FKK}), though
this is the case for the blow-analytic equivalence by T.-C. Kuo
\cite{Ku} (another advantage of the definition of the blow-analytic
equivalence by T.-C. Kuo is that for an analytic function $f$ on an analytic manifold $M$ there exists a real modification $\beta:M_f\to M$ such that $f\circ\beta$ has only normal crossing singularities, by Hironaka Desingularization Theorem). The
problem comes from transitivity. However, if we admit blowings-up along non-smooth analytic centers
 in the definitions of Blow-analytic equivalence and almost
 Blow-analytic equivalence, we can face this issue.

\begin{prop}\label{eqrel} Allowing non-smooth analytic centers
 in the definitions, Blow-analytic equivalence and almost
 Blow-analytic equivalence become equivalence relations.
\end{prop}

\begin{proof} Let $\mathcal O_{M_1}$ denote the sheaf of analytic function germs on an analytic space $M_1$. 
For a morphism $g:M_1\to M_2$ of analytic spaces, as locally ringed spaces, and for a sheaf of $\mathcal O_{M
_2}$-ideals $\mathcal I$, let $g^{-1}\mathcal I\!\cdot\!\mathcal O_{M_1}$ denote the inverse image ideal sheaf, i.e., 
the sheaf of $\mathcal O_{M_1}$-ideals generated by the image of the inverse image $g^{-1}\mathcal I$ of 
$\mathcal I$ (see \cite{Ht}). 
Let $f_i,\ i=1,..,4$, be analytic functions on an analytic manifold
$M$ with $f_1$ almost Blow-analytically equivalent to $f_2$, $f_3$
almost Blow-analytically equivalent to $f_4$ and $f_2=f_3$. Namely, let 
$$N_{i,k_i}
\overset {\pi_{i,k_i}}
\longrightarrow\cdots\longrightarrow
N_{i,1}\overset{\pi_{i,1}}\longrightarrow N_{i,0}=M, \text{~~~for~~~}i=1,\ldots,4$$
be sequences of blowings-up with respect to coherent sheaves of non-zero $\mathcal O_{N_{i,j}}$-ideals, and 
let $\tau_2:N_{1,k_1}\to N_{2,k_2}$ and $\tau_4:N_{3,k_3}\to N_{4,k_4}$ be isomorphisms of locally 
ringed spaces such that 
$$f_i\circ\pi_{i,1}\circ\cdots\circ\pi_{i,k_i}\circ\tau_i=f_{i-1}\circ\pi_{i-1,1}
\circ\cdots\circ\pi_{i-1,k_{i-1}},$$ $i=2,4$. 
If the compositions
$$N_{2,k_2}\overset
{\pi_{2,k_2}}\longrightarrow\cdots\longrightarrow M ~~~\text{and~~~}N_{3,k_3}\overset{\pi_{
3,k_3}}\longrightarrow\cdots\longrightarrow M$$
coincide each other, then $\tau_4\circ\tau_2$ is an 
isomorphism from $N_{1,k_1}$ to $N_{4,k_4}$ and 
$$f_4\circ\pi_{4,1}\circ\cdots\circ\pi_{4,k_4}\circ\tau_4
\circ\tau_2=f_1\circ\pi_{1,1}\circ\cdots\circ\pi_{1,k_1}.$$ 

Hence it suffices to reduce the problem to the case where
$$N_{2,k_2}\overset {\pi_{2,k_2}}\longrightarrow
\cdots\longrightarrow M \text{~~~and~~~}N_{3,k_3}\overset {\pi_{
3,k_3}}\longrightarrow\cdots\longrightarrow M$$
coincide. 

We will use the following fact. 
Let $\mathcal J_1$ and $\mathcal J_2$ be coherent sheaves of non-zero $\mathcal O_{M_1}$-ideals on a reduced and 
irreducible analytic space $M_1$, and let $g_i:L_i\to M_1$ denote the blowing-up with respect to $\mathcal J_i$
for $i=1,2$. Let $h_1:N_1\to L_1$ and $h_2:N_2\to L_2$  denote the blowings-up with respect to $\ g_1^{-1}\mathcal J_2\!\cdot\!\mathcal O
_{L_1}$ and $g_2^{-1}\mathcal J_1\!\cdot\!\mathcal O_{L_2}$,
respectively. 
Then there exists an unique isomorphism $\tau:N_1\to N_2$ such that $g_1\circ h_1=g_2\circ h_2\circ
\tau$.

$$\xymatrix{
N_1\overset\tau\cong N_2  \ar[r]^{h_2}\ar[d]_{h_1} &L_2 \ar[d]_{g_2} \\
L_1\ar[r]^{g_1} &M_1 \\ 
}$$

Actually, by the universal property theorem of the blowing-up (see \cite{Ht} in
the algebraic case) apply to the blowing-up 
$g_2:L_2\to M_1$ and the morphism $g_1\circ h_1:N_1\to M_1$, there exists an unique morphism $\pi:N_1\to L_2$ such that $g_2\circ\pi=g_1\circ h_1$ since 
$$(g_1\circ h_1)^{-1}\mathcal J_2\!\cdot\!\mathcal O_{N_1}\,(=h^{-1}_1(g_1^{-1}\mathcal J_2\!\cdot\!\mathcal O_{L_1})\!
\cdot\!\mathcal O_{N_1})$$ 
is invertible. 
Next, considering the blowing-up $h_2:N_2\to L_2$ and the morphism $\pi:N_1\to L_2$, we obtain a unique 
morphism $\tau:N_1\to N_2$ such that $\pi=h_2\circ\tau$ since 
$$\pi^{-1}(g_2^{-1}\mathcal J_1\!\cdot\!
\mathcal O_{L_2})\!\cdot\!\mathcal O_{N_2}\,(=(g_2\circ\pi)^{-1}\mathcal J_1\!\cdot\!\mathcal O_{N_2}=(g_1\circ h_1)
^{-1}\mathcal J_1\!\cdot\!\mathcal O_{N_2}=h^{-1}_1(g_1^{-1}\mathcal
J_1\!\cdot\!\mathcal O_{L_1})\mathcal O_{N_2})$$ 
is 
invertible. 
Then $$g_1\circ h_1=g_2\circ\pi=g_2\circ h_2\circ\tau.$$ 
By the same reason, we have an unique morphism $\tau':N_2\to N_1$ such that $g_1\circ h_1\circ\tau'=g_
2\circ h_2$. 
Hence $\tau$ is an isomorphism.

By this fact, we obtain a commutative diagram of blowings-up with respect to coherent sheaves of $\mathcal O_
{M_{i,j}}$-ideals\,:

$$\xymatrix{
M_{k_2,k_3} \ar[r]^{\nu_{k_2,k_3}}\ar[d]_{\mu_{k_2,k_3}} & M_{k_2-1,k_3} \ar[r]\ar[d]_{\mu_{k_2-1,k_3}} & \cdots \ar[r]^{}& M_{0,k_3} \ar[d]_{\mu_{0,k_3}}\\
M_{k_2,k_3-1} \ar[r]^{\nu_{k_2,k_3-1}}\ar[d]_{} &  \ar[r]\ar[d]_{} &
\cdots \ar[r]^{}& M_{0,k_3-1} \ar[d]_{}\\
\vdots \ar[d]_{} & \vdots \ar[d]_{} & & \vdots\ar[d]_{}\\
M_{k_2,0} \ar[r]^{\nu_{k_2,0}} & M_{k_2-1,0} \ar[r] & \cdots \ar[r]^{}& M_{0,0}\\
}$$
such that 
$$M_{k_2,0}\overset
{\nu_{k_2,0}}\longrightarrow\cdots\overset {\nu_{1,0}}\longrightarrow M_
{0,0} \text{~~~and~~~} M_{0,k_3}\overset
{\mu_{0,k_3}}\longrightarrow\cdots\overset {\mu_{0,1}}\longrightarrow M
_{0,0}$$ 
coincide with 
$$N_{2,k_2}\overset
{\pi_{2,k_2}}\longrightarrow\cdots\overset {\pi_{2,1}}\longrightarrow
M \text{~~~ and~~~} N_{3,k_3}\overset
{\pi_{3,k_3}}\longrightarrow\cdots\overset {\pi_{3,1}}\longrightarrow
M$$ 
respectively. 
In particular, 
$$
\pi_{3,1}\circ\cdots\circ\pi_{3,k_3}\circ\nu_{1,k_3}\circ\cdots\circ\nu_{k_2,k_3}=\pi_{2,1}\circ\cdots
\circ\pi_{2,k_2}\circ\mu_{k_2,1}\circ\cdots\circ\mu_{k_2,k_3}. 
$$
Since $\tau_2:N_{1,k_1}\to M_{k_2,0}$ is an isomorphism, we obtain a
sequence of blowings-up 
$$N_{1,k_1+k_3}
\overset {\pi_{1,k_1+k_3}}\longrightarrow\cdots\overset {\pi_{1,k_1+1}}\longrightarrow N_{1,k_1}$$
and several isomorphisms 
$$\tau_{2,k_3}:N_{1,k_1+k_3}\to M_{k_2,k_3},~~\ldots~,~~\tau_{2,0}=\tau_2:N_{1,
k_1}\to M_{k_2,0}$$ 
such that the following diagram is commutative.

$$\xymatrix{
N_{1,k_1+k_3} \ar[r]^{\pi_{1,k_1+k_3}}\ar[d]_{\tau_{2,k_3}} &
N_{1,k_1+k_3-1} \ar[r]\ar[d]_{\tau_{2,k_3-1}} & \cdots
\ar[r]^{\pi_{1,k_1+1}}& N_{1,k_1} \ar[d]_{\tau_{2,0}}\\
M_{k_2,k_3} \ar[r]^{\mu_{k_2,k_3}} & M_{k_2,k_3-1}\ar[r] & \cdots
\ar[r]^{\mu_{k_2,1} } & M_{k_2,0} .\\ 
}$$
Then the following equalities hold:

$$
f_2\circ\pi_{2,1}\circ\cdots\circ\pi_{2,k_2}\circ\mu_{k_2,1}\circ\cdots\circ\mu_{k_2,k_3}\circ\tau_{2,
k_3}=\cdots=$$
$$
f_2\circ\pi_{2,1}\circ\cdots\circ\pi_{2,k_2}\circ\tau_{2,0}\circ\pi_{1,k_1+1}\circ\cdots\circ\pi_{1,k_
1+k_3}=$$
$$ f_1\circ\pi_{1,1}\circ\cdots\circ\pi_{1,k_1}\circ\pi_{1,k_1+1}\circ\cdots
\circ\pi_{1,k_1+k_3}. 
$$

In the same way we obtain a commutative diagram:

$$\xymatrix{
M_{k_2,k_3} \ar[r]^{\nu_{k_2,k_3}}\ar[d]_{\tau_{4,k_2}} &M_{k_2-1,k_3} \ar[r]\ar[d]_{\tau_{4,k_2-1}} & \cdots \ar[r]^{\nu_{1,k_3}}& M_{0,k_3}\ar[d]_{\tau_{4,0}}\\
N_{4,k_2+k_4} \ar[r]^{\pi_{4,k_2+k_4}} & N_{4,k_2+k_4-1} \ar[r] & \cdots \ar[r]^{\pi_{4,k_4+1}} &N_{4,k_4} \\ 
}$$
where the horizontal morphisms are blowings-up and the vertical ones are isomorphisms, with $\tau_{4,0}
=\tau_4$, and 

$$f_4\circ\pi_{4,1}\circ\cdots\circ\pi_{4,k_4}\circ\pi_{4,k_4+1}\circ\cdots\circ\pi_{4,k_2+k_4}\circ\tau_
{4,k_2}=$$
$$f_3\circ\pi_{3,1}\circ\cdots\circ\pi_{3,k_3}\circ\nu_{1,k_3}\circ\cdots\circ\nu_{k_2,k_3}.$$

Therefore

$$
f_4\circ\pi_{4,1}\circ\cdots\circ\pi_{4,k_2+k_4}\circ\tau_{4,k_2}\circ\tau_{2,k_3}=f_3\circ\pi_{3,1}
\circ\cdots\circ\nu_{k_2,k_3}\circ\tau_{2,k_3}=$$
$$
f_2\circ\pi_{2,1}\circ\cdots\circ\mu_{k_2,k_3}\circ\tau_{2,k_3}=f_1\circ\pi_{1,1}\circ\cdots\pi_{1,k_1+
k_3}
$$
which proves the transitivity for almost Blow-analytic
equivalence. Considering Blow-analytic
equivalence, i.e. if moreover $\tau_2$ and $\tau_4$ induce homeomorphisms of $M$, then $\tau_{4,k_2}\circ\tau_{2,k_3}$ induces automatically a homeomorphism of $M$. 
\end{proof}


\subsection{Blow-analytic versus almost Blow-analytic equivalence}\label{sect2}
In $\R^2$, blow-analytic equivalence coincides with almost
Blow-analytic equivalence. We prove in this section that these
relations do not coincide in general, by giving an explicit example in
$\R^4$ of almost Blow-analytically equivalent functions whose germs at
$0$ are not blow-analytically equivalent.

Let $f$ and $g$ be the $C^\omega$ functions on $\R^4$ in variables $(u,v,w,x)$ defined by 
$$
f=\phi\psi\xi,~~~~~ g=\phi\psi\eta,$$
where
$$
\phi=u^2+v^2,~~~~~ \psi=u^4+v^2+u^2w^2,$$
$$
\xi=u^4+v^2+u^2(w-x)^2,~~~~~ \eta=u^4+(v-xu)^2+u^2w^2.
$$

\begin{lemma} The functions $f$ and $g$ are almost Blow-analytically equivalent.
\end{lemma}

\begin{proof}Set $X=\{0\}\times\R^2\subset\R^4$. 
Let $\Lambda$ denote the set of half-lines in $\R^4$ starting from
points in $X$ and orthogonal to $X$. Denote by $e(\lambda)$ the
endpoint of $\lambda\in\Lambda$, namely $e(\lambda)=\lambda\cap
X$. Note that the functions $f|_\lambda$ and $g|_\lambda$ have singularities 
only at $e(\lambda)$ for any $\lambda\in\Lambda$, and that $f(\lambda)=g(\lambda)=[0,\,\infty)$.

We will prove the almost Blow-analytic equivalence of $f$ and $g$ in
such a way that the diffeomorphism that will realize the equivalence, induces a diffeomorphism 
of $\R^4-X$ carrying any $\lambda-e(\lambda)$, with
$\lambda\in\Lambda$, to some $\lambda'-e(\lambda')$, with $\lambda'\in\Lambda$.

Let $\pi:M\to\R^4$ denote the blowing-up along center $X$. 
Then 
$$
M=\{(s:t,u,v,w,x)\in\PP(1)\times\R^4:sv=tu\},$$
$$
f\circ\pi(1:t,u,v,w,x)=u^6(1+t^2)(u^2+t^2+w^2)(u^2+t^2+(w-x)^2),$$
$$
g\circ\pi(1:t,u,v,w,x)=u^6(1+t^2)(u^2+t^2+w^2)(u^2+(t-x)^2+w^2),$$
$$
f\circ\pi(s:1,u,v,w,x)=v^6(1+s^2)(s^4v^2+1+s^2w^2)(s^4v^2+1+s^2(w-x)^2),$$
$$
g\circ\pi(s:1,u,v,w,x)=v^6(1+s^2)(s^4v^2+1+s^2w^2)(s^4v^2+(1-xs)^2+s^2w^2)
$$
and for each $\lambda\in\Lambda$, the set
$\pi^{-1}(\lambda-e(\lambda))$ is defined by
\begin{equation}\tag{*}
\{(1:t_0,u,t_0u,w_0,x_0):u\in(0,\,\infty)\}
\end{equation}
or 
$$\{(s_0:1,s_0v,v,w_0,x_0):v\in(0,\,\infty)\},$$
depending on the chart, 
for some  $s_0,t_0,w_0,x_0\in\R$. Hence
$$\Sing f\circ\pi=\Sing g\circ\pi=\pi^{-1}(X)$$
and the germs of $f\circ\pi$ and $g\circ\pi$ at points of $\pi^{-1}(X)-M_1-M_2$ and $\pi^{-1}(X)-
M_1-M_3$, respectively, are sixth powers of some regular function germs, where 
$$
M_1=\{(1:0,0,0,0,x)\in M\},$$
$$
M_2=\{(1:0,0,0,w,x)\in M:w=x\},$$
$$
M_3=\{(1:t,0,0,0,x)\in M:t=x\}.
$$

We will construct below a $C^\infty$ diffeomorphism $h$ of $M$ such
that $f\circ\pi$ is equal to $g\circ\pi\circ 
h$, the image of $M_1\cup M_2$ under $h$ is equal to $M_1\cup M_3$ and $h$ is of class $C^\omega$ on a neighborhood of $M_1\cup 
M_2$. Assuming the existence of such a diffeomorphism, we obtain by easy calculations two compositions 
of two sequences of blowings-up $\pi_f:N\to M$ and $\pi_g:L\to M$ along smooth analytic centers 
and a $C^\infty$ diffeomorphism $\tilde h:N\to L$ such that the union of centers in $M$ and the images of the centers is $M_1\cup M_2\cup M_3$, $f\circ\pi\circ\pi_f=g\circ\pi\circ\pi
_g\circ\tilde h$ and the functions $f\circ\pi\circ\pi_f$ and $g\circ\pi\circ\pi_g$ have only normal crossing 
singularities. Therefore $f\circ\pi\circ\pi_f$ and $g\circ\pi\circ\pi_g$ 
are $C^\omega$ right equivalent by Theorem 3.1,(1) in \cite{FS}, i.e. $f$ and $g$ are almost Blow-analytically equivalent.

Now we give the construction of the expected $C^\infty$ diffeomorphism $h$.
Denote by $(t,w)$ the variables of $\R^2$, and for each $x\in\R$, let $B_{|x|}$ denote the ball 
in $\R^2$ with center 0 and radius $|x|$. 
Let $h_x$ be a $C^\infty$ diffeomorphism of $\R^2$ which is identical
outside of $B_{2+|x|}$, with
$h_x(t,w)=(w,-t)$ on $B_{1+|x|}$ and such that the map 
$$\R^3 \ni (t,w,x) \to h_x(t,w) \in \R^2$$
is of class $C^\infty$. Define the map $h:\R^4\to\R^4$ by $h(u,t,w,x)=(u,h_x(t,w),
x)$. Then $h$ is a $C^\infty$ diffeomorphism. 
Regarding $\R^4$ as a subset of $M$ by the map 
$$(u,t,w,x)\to(1:t,u,tu,w,x),$$ 
we consider $h$ defined on $\R^4$ in $M$ and we extend it to the whole
of $M$ by
the identity map. Therefore we obtain a $C^\infty$ diffeomorphism $\tilde h$ of
$M$ such that 
\begin{itemize}
\item $\tilde h(\pi^{-1}(X))=\pi^{-1}(X)$, 
\item the image $U$ of $\R\times\cup_{x\in\R}(B_{1+|x|}\times\{x\})$ in $M$ under the inclusion 
map is a neighborhood of $M_1\cup M_2\cup M_3$, 
\item $\tilde h$ is of class $C^\omega$ on $U$, 
\item $\tilde h(M_1)=M_1$, 
\item $\tilde h(M_2)=M_3$, 
\item for any 
$\lambda\in\Lambda$ there exists $\lambda'$ by $(*)$ such that $\tilde h(\pi^{-1}(\lambda-e(\lambda))
=\pi^{-1}(\lambda'-e(\lambda'))$, 
\item $\tilde h$ is of class $C^\omega$ on $U$, 
\item and $f\circ\pi/(1+t^2)=g\circ\pi
\circ\tilde h/(1+w^2)=u^6(u^2+t^2+w^2)(u^2+t^2+(w-x)^2)$ on $U$.
\end{itemize}
We modify $\tilde h$ so that $f\circ\pi=g\circ\pi
\circ\tilde h$ on $U$ as follows. 
Set 
$$
h_1(u,t,w,x)=(1+w^2)^{1/10}(u,t,w,x),
$$
$$
h_2(u,t,w,x)=(1+t^2)^{1/10}(u,t,w,x), 
$$
let $h_3$ be a $C^\infty$ diffeomorphism of $\R^4$ which equals $h_1^{-1}\circ h_2$ on $\R\times\cup_{x\in\R}(B_{\alpha(x)}\times\{x\})$ and the identity map outside of $\R\times\cup_{x\in\R}(B_{\beta(x)}\times\{x\})$ for some positive $C^\infty$ functions $\alpha$ and $\beta$ on $\R$, and extend $h_3$ to a $C^\infty$ diffeomorphism $\tilde h_3$ of $M$ as above. 
Then we can choose $h_1,\ h_2,\ \alpha$ and $\beta$ so that $\tilde h\circ\tilde h_3$ satisfies the above conditions on $\tilde h$ except the last, and the last becomes  
$$
f\circ\pi=g\circ\pi\circ\tilde h\circ\tilde h_3\quad\text{on}\ U
$$
We replace $\tilde h$ with $\tilde h\circ\tilde h_3$ and write $\tilde h\circ\tilde h_3$ as $\tilde h$. 
Thus we have $f\circ\pi=g\circ\pi\circ\tilde h$ on $U$. \par
It remains only to modify $\tilde h$ outside of $U$ so that the
equality $f\circ\pi=g\circ\pi\circ\tilde h$ holds everywhere. 
For any $a\in M$, define $\tilde {\tilde h}(a)$ by  $$\tilde {\tilde h}(a)=\overline{\pi^{-1}(\lambda-e(\lambda))}\cap(g
\circ\pi)^{-1}(f\circ\pi(a))$$
for some $\lambda\in\Lambda$ satisfying $\tilde h(a)\in\overline{\pi^{
-1}(\lambda-e(\lambda))}$. 
Then $\tilde {\tilde h}$ is a well-defined $C^\infty$ diffeomorphism of $M$ since $$\tilde {\tilde h}(a)
=\overline{\pi^{-1}(\lambda-e(\lambda))}\cap(g\circ\pi)^{-1}(f\circ\pi(a))$$ 
for $a\in U$, the set $U$ 
is a union of some $\overline{\pi^{-1}(\lambda-e(\lambda))}$, $\lambda\in\Lambda$, and because the 
restrictions of $f\circ\pi$ and $g\circ\pi$ to each $\overline{\pi^{-1}(\lambda-e(\lambda))}$ 
outside of $U$ are $C^\infty$ right equivalent to the function $$[0,\,\infty)\ni z\to z^6\in\R.$$ 
Now the equality $f\circ\pi=g\circ\pi\circ\tilde {\tilde h}$ is
satisfied on $M$.
\end{proof}

\begin{lemma}The germs of $f$ and $g$ at the origin are not blow-analytically 
equivalent.
\end{lemma}

We have to prove that there do not exist open neighborhoods $U$ and
$V$ of the origin in $\R^4$ such that $f|_U$ 
and $g|_V$ are blow-analytically equivalent and the induced
homeomorphism from $U$ to $V$ fixes the origin.

\begin{proof} We introduce an invariant of blow-analytic equivalence, which is a generalisation of one in \cite{Fu}. 
Let $C$ denote the set of germs at the origin of analytic curves $c:[0,\,\epsilon)\to\R^4$, with
$\epsilon>0$. We give a topology on $C$ by identifying $C$ with 4-product of the one-variable 
convergent power series ring $\R\langle\langle t\rangle\rangle^4$ and choosing the product topology 
on $\R\langle\langle t\rangle\rangle$, i.e.
$$\R\langle\langle t\rangle\rangle\ni\sum_{n=0}^\infty  a_{n,k}t^n\to
0 \text{~~~as~~~} k \to \infty$$ if for any $n=0,\ldots,
\infty$
$$ \R\ni a_{n,k}\to 0 \text{~~~as~~~} k\to\infty.$$ 
Then $p:C\to\R^4$, defined by $p(c)=c(0)$, is a topological fibre
bundle.

Let $\pi:M\to\R^4$ be the composition of a finite sequence of blowings-up of $\R^4$ along smooth analytic 
centres. 
The map $\pi$ naturally induces a surjective $C^0$ map $\pi_*:C_M\to C$ such that $p\circ\pi_*=\pi
\circ p_M$, where $C_M$ is the analytic curve germs in $M$ and $p_M:C_M\to M$ is defined by $p_M
(c)=c(0)$.

Let us assume that the germs of $f$ and $g$ at 0 are blow-analytically
equivalent. Then there exist open 
neighbourhoods $U$ and $V$ of 0 in $\R^4$ and a map
$$
\tau:\{c\in C:c(0)\in U,\ f\circ c\not\equiv 0\}\to\{c\in C:c(0)\in V,\ g\circ c\not\equiv 0\}
$$
such that $f\circ c=g\circ\tau(c)$ as
analytic function germs in the variable $t$ for $c\in p^{-1}(U)$. 
Let $\tau_0:U\to V$ denote the homeomorphism such that $p\circ\tau=\tau_0\circ p$.

We denote by $o_f(c)$ the order of $f\circ c$ at 0 for $c\in C$, so that we define a map $o_f:C\to\N\cup\{\infty\}$.
Set $C_a=p^{-1}(a)$, for $a\in U$, and consider the family $\{C_a\cap o^{-1}_f(i):i\in\N\}$. 
We stratify $U$ by the groups of the family. 
Set 
$$
X_1=\{a\in U:o_f=0\ \text{on}\ C_a\}$$
Then
$$X_1=\{(u,v,w,x)\in U:u\not=0\ \text{or}
\ v\not=0\}. $$
Set also
$$X_2=\{u=v=0,w\not=0,w\not= x\},$$
$$X_3=\{u=v=w=0,x\not= 0\}\cup\{u=v=0,w=x\not=0\},$$
$$X_4=\{u=v=w=x=0\}.$$

Then $\{X_1,X_2,X_3,X_4\}$ is a stratification of $U$. Moreover if
$a$ belongs to $X_2$, then $(C_a,C_a\cap o_f^{-1}(6))$ has the same homotopy groups as $([-1,\,1]^2,\partial([-1,\,1]^2))$ because 
$$
C_a\cap o_f^{-1}(6)=\{c=(c_1,..,c_4):[0,\,\epsilon)\to\R^4\in C_a:\frac{d c_1}{d t}(0)\not=0\ 
\text{or}\ \frac{d c_2}{d t}(0)\not=0\}.$$
Similarly if $a$ belongs to $X_3$ or $X_4$, then $(C_a,C_a\cap
o_f^{-1}(6))$ has the same homotopy groups as $([-1,\,1]^2,(\partial[-1,\,1])^2)$ because

$$C_a\cap o_f^{-1}(6)=\{\frac{d c_1}{d t}(0)\not=0,\ \frac{d c_2}{d t}(0)\not=0\} .$$

Now we consider $g$. 
Define $o_g$ in the same way and set 
$$Y_1=\{a\in V:o_g=0\ \text{on}\ C_a\}=\{(u,v,w,x)\in V:u\not=0\
\text{or}\ v\not=0\},$$
$$Y_2\!=\!\{u=v=0,w\not=0\},$$
$$Y_3\!=\!\{u=v=w=0,x\not=0\},$$
$$Y_4\!=\!\{u=v=w=x=0\}.$$
If $a\in Y_1,\,Y_2,\,Y_3$ or $Y_4$, then $(C_a,C_a\cap o_g^{-1}(6))$ has the same homotopy groups as $([-1,\,1]^2,\emptyset)$, $([-1,\,1]^2,\partial([-1,\,1]^2))$, $([-1,\,1]^2,\{6\ \text{points in }\partial([-1,\,1]^2)\})$ or $([-1,\,1]^2,(\partial[-1,\,1])^2)$ respectively. 
We omit its proof except for $Y_3$ because the other case is treated in the same way as above. 
Let $a\in Y_3$. 
Then 
$$
C_a\cap o_g^{-1}(6)=\{\frac{d c_1}{d t}(0)\not=0,\ \frac{d c_2}{d t}(0)\not=0,\ x\frac{d c_1}{d t}(0)\not=\frac{d c_2}{d t}(0)\}.
$$
Hence $(C_a,C_a\cap o_g^{-1}(6))$ has the same homotopy groups as $([-1,\,1]^2,\{6\ \text{points in}\linebreak\partial([-1,\,1]^2)\})$. 
Note in any case $\frac{d c_1}{d t}(0)\not=0$ or $\frac{d c_2}{d t}(0)\not=0$ for $c=(c_1,..,c_4)\in C_a\cap o^{-1}_g(6)$. 
Thus $\{Y_1,Y_2,Y_3,Y_4\}$ is a stratification of $V$ by the homotopy groups of
$C_a\cap o_g^{-1}(6)$.

On the other hand, $o_f(c)$ is equal to $o_g(\tau(c))$ for any $c\in C$ with $p(c)\in U$ and $o_f(c)<\infty$, since $f$ and $g$ are blow-analytically equivalent. 
Hence for $a\in U$, the maps $o_f:\{c\in C_a:o_f(c)<\infty\}\to\N$ and $o_g:\{c\in C_{\tau_0(a)}:o_g(c)<\infty\}\to\N$ are $C^0$ right equivalent. 
Moreover, the restriction of $\tau$ to $\{c\in C_a:o_f(c)=6\}$ is a homeomorphism onto $\{c\in C_{\tau_0(a)}:o_g(c)=6\}$. 
For that it suffices to prove the following statement. 

Let $M\overset{\pi_1}\to M_1\to\cdots\overset{\pi_k}\to M_k=\R^4$ be a sequence of blowings-up along smooth analytic centers such that the images of centers in $\R^4$ are included in $\{u=v=0\}$. 
Set $\pi=\pi_k\circ\cdots\circ\pi_1$. 
Then the following maps are homeomorphisms: 
$$
\pi_*^{-1}\{c\in C_a:o_f(c)=6\}\overset{\pi_{1*}}\longrightarrow\cdots\overset{\pi_{k*}}\longrightarrow\{c\in C_a:o_f(c)=6\},
$$
$$
\pi_*^{-1}\{c\in C_{\tau_0(a)}:o_g(c)=6\}\overset{\pi_{1*}}\longrightarrow\cdots\overset{\pi_{k*}}\longrightarrow\{c\in C_{\tau_0(a)}:o_g(c)=6\}.
$$

Consider only the maps in the former sequence. 
They are clearly continuous and bijective. 
Hence we show the inverse maps are continuous. 
Let $c=(c_1,..,c_4)\in C_a$ with $o_f(c)=6$. 
Then $\frac{d c_1}{d t}(0)\not=0$ or $\frac{d c_2}{d t}(0)\not=0$ by the above arguments. 
Namely $c$ is transverse to $\{u=v=0\}$ and hence to the center of $\pi_k$. 
Hence the inverse map of the map $\pi_{k*}^{-1}\{c\in C_a:o_f(c)=6\}\to\{c\in C_a:o_f(c)=6\}$ is continuous by the definition of a blowing-up. 
It follows also that for that $c$, $\pi_{k*}^{-1}(c)$ is transverse to each irreducible component of $\pi_k^{-1}\{u=v=0\}$. 
Thus repeating the same arguments we see the inverse maps are continuous. 

Therefore, $\{c\in C_a:o_f(c)=6\}$ and $\{c\in C_{\tau_0(a)}:o_g(c)=6\}$ have the same homotopy groups. 
Hence 
$$
\tau_0(X_1)=Y_1,\ \tau_0(X_2)=Y_2,\ \tau_0(X_3\cup X_4)=Y_4, 
$$
which contradicts the assumption that $\tau_0$ is a homeomorphism.
\end{proof}

\begin{rmk}
We can prove moreover that $f$ and $g$ are almost Blow-Nash
equivalent. Actually, similarly to the preceding proof, we can find compositions of finite sequences of blowings-up of 
$\R^4$ along smooth Nash centers $\pi_f:N\to\R^4$ and $\pi_g:L\to\R^4$
so that $f\circ\pi_f$ and 
$g\circ\pi_g$ are semi-algebraically $C^m$ right equivalent for any $m\in\N$ by using a partition of 
unity of class semi-algebraic $C^m$, \S\,II.2, \cite{Shiota}. We may
conclude the statement by theorem 3.1,(3) in \cite{FS}.
\end{rmk}

We have finally proved:

\begin{prop} Almost Blow-analytically equivalent function germs are not 
necessarily blow-analytically equivalent.
\end{prop}

\subsection{Normal crossings and cardinality}

An analytic function with {\it only normal crossing singularities} at a point $x$ of a manifold is 
a function whose germ at $x$ is of the form $$\pm x^\alpha
(=\pm\prod_{i=1}^n x_i^{\alpha_i})\,+\const, \text{~~~}
\alpha=(\alpha_1,...,\alpha_n)\not=0\in\N^n$$ 
for some local analytic coordinate system $(x_1,...,x_n)$ at 
$x$. 
If the function has only normal crossing singularities everywhere, we say the function has {\it only normal 
crossing singularities}. 
By Hironaka Desingularization Theorem, any analytic function becomes one with only normal 
crossing singularities after a finite sequence of blowings-up along
smooth centers in the local case or in the compact case.  
An analytic subset of an analytic manifold is called {\it normal crossing} if it is the zero set of an 
analytic function with only normal crossing singularities. 
This analytic function is called {\it defined by} the analytic set. 
It is not unique. 
However, the sheaf of $\mathcal O$-ideals {\it defined by} the analytic set is naturally defined and unique. 
We can naturally stratify a normal crossing analytic subset $X$ into analytic manifolds $X_i$ of dimension 
$i$. 
We call $\{X_i\}$ the {\it canonical stratification} of $X$.

The difference between almost Blow-analytic equivalence and almost 
Blow-analytic R-L equivalence is tiny.

\begin{prop} If the germs of $f$ and $g$ on a compact connected component of $\Sing f$ in $M$ are almost 
Blow-analytically R-L equivalent then the germs of $f$ and $g+\const$ or of $f$ and $-g+\const$ are 
almost Blow-analytically equivalent.
\end{prop}

\begin{proof} We can reduce the problem to the case where $f\circ\pi_f$ and $g\circ\pi_g$ have only
normal crossing singularities by Hironaka Desingularization
Theorem. Then the result follows from the remark following lemma 4.7 and proposition 4.8,(i) in \cite{FS}.
\end{proof}

Let $P$ denote the set of homogeneous polynomial functions on $\R^2$ of degree 4. 
It is easy to see that the analytic R-L equivalence classes of $P$ has the cardinality of the continuum, and the Blow-analytic R-L equivalence classes of $P$ is 
finite. Moreover almost Blow-analytic R-L equivalence on analytic functions on $\R^2$ coincides with Blow-analytic R-L equivalence and is an equivalence relation.

For general dimension, we do not know whether almost Blow-analytic 
(Blow-Nash) R-L equivalence is an equivalence relation, as noticed
before. Therefore we will say that analytic (Nash) functions $f$ and $g$ lie in the {\it same class} if there exists a sequence 
of analytic functions $f_0,...,f_k$ such that $f_0=f,\,f_k=g$ and $f_i$ and $f_{i+1}$ are almost 
Blow-analytically (Blow-Nash) R-L equivalent, $i=0,...,k-1$.

In that setting, we are able to determine the cardinality of the
classes under these relations.

\begin{thm}\label{main} Let $M$ be a compact analytic (resp. Nash) manifold of strictly positive dimension. 
Then the cardinality of the set of classes of analytic (resp. Nash) functions on $M$, classified by 
almost Blow-analytic (resp. Blow-Nash) R-L equivalence, is
countable. Moreover, in the Nash case, we do not need to assume that $M$ is compact. 
\end{thm}

We proved in \cite{FS} that the cardinality of analytic R-L
equivalence classes of analytic functions on $M$ with only normal
crossing singularities is 0 or countable. In particular, it can not be
0 in theorem \ref{main} since we can produce non-trivial examples. We
proved moreover in \cite{FS} that in
the Nash case, the result holds true even if $M$ is no longer
compact, and in that case the cardinality is always
countable. Combined with Hironaka Desingularization Theorem, we obtain
a proof for theorem \ref{main}.


\subsection{Nash Approximation Theorems}
The approximation of analytic solutions of a system of Nash equations
is a crucial tool in the study of real analytic and Nash functions on
Nash manifolds. A local version is given by the classical Artin
approximation theorem. The Nash
Approximation Theorem of \cite{CRS} states a global version on a
compact Nash manifold.

The following result is a natural counterpart, for almost
Blow-analytic equivalence, of the Nash
Approximation Theorem of \cite{CRS}. The remaining part of the paper will be devoted
to its proof.

\begin{thm}\label{main}
Let $M$ be a Nash manifold, $X\subset M$ be a compact semialgebraic
subset and $f,g$ Nash function germs on $X$ in $M$ such that $X=M$ or $X\subset
\Sing f$. 
If $f$ and $g$ are almost Blow-analytically (R-L) equivalent, then $f$ and $g$ are almost Blow-Nash 
(respectively R-L) equivalent.
\end{thm}

\begin{rmk}\begin{flushleft}\end{flushleft}
\begin{enumerate}
\item Here the compactness assumption on $X$ is necessary. 
Indeed, there exist a non-compact Nash manifold $M$ and Nash functions $f$ and $g$ on $M$ which are 
$C^\omega$ right equivalent but not almost Blow-Nash equivalent as follows. 
Let $N$ be a compact contractible Nash manifold with non-simply connected boundary of 
dimension $n>3$ (e.g., see \cite{Mz}). 
Set $M=(\Int N)\times(0,\,1)$ and let $f:M\to(0,\,1)$ denote the projection. 
Then $M$ and $f$ are of class Nash, and $M$ is Nash diffeomorphic to $\R^{n+1}$ for the 
following reason. 
Smooth the corners of $N\times[0,\,1]$. 
Then $N\times[0,\,1]$ is a compact contractible Nash manifold with simply connected boundary 
of dimension$\,>4$. 
Hence by the positive answers to Poincar\'e conjecture and Sch\"onflies problem (Brown-Mazur 
Theorem) $N\times[0,\,1]$ is $C^\infty$ diffeomorphic to an $(n+1)$-ball. 
Hence by Theorem VI.2.2, \cite{Shiota} $M$ is Nash diffeomorphic to an open $(n+1)$-ball. 
Let $g:M\to\R$ be a Nash function which is Nash right equivalent to the projection $\R^n
\times(0,\,1)\to(0,\,1)$. 
Then $f$ and $g$ are $C^\omega$ right equivalent since $\Int N$ is $C^\omega$ diffeomorphic 
to $\R^n$, but they are not almost Blow-Nash equivalent because if they were so then their 
levels would be Nash diffeomorphic except for a finite number of values and hence $\Int N$ and 
$\R^n$ are Nash diffeomorphic, which contradicts Theorem VI 2.2, \cite{Shiota}. 
\item The similar result concerning Blow-Nash equivalence remains open. Namely we do not know whether Blow-analytically equivalent Nash function germs on $X$ in $M$ are Blow-Nash equivalent.
\end{enumerate}
\end{rmk}


\section{Nash approximation of an analytic desingularization}

\subsection{Preliminaries on real analytic sheaf theory}

We recall the statements of the real analytic case of Cartan Theorems A and B, and Oka Theorem, in the refined version given in \cite{FS}.

Let $\mathcal O$, $\mathcal N$ and $N(M)$ denote, respectively, the sheaves of analytic and Nash 
function germs on an analytic and Nash manifold and the ring of Nash functions on a Nash manifold 
$M$. 
We write $\mathcal O_M$ and $\mathcal N_M$ when we emphasize the domain $M$. 
For a function $f$ on an analytic (Nash) manifold $M$, a subset $X$ of $M$, a vector field $v$ on $M$ 
and for a sheaf of $\mathcal O$- ($\mathcal N$-) modules $\mathcal M$ on $M$, let $f_x$, $X_x$, $v_x$ and $\mathcal M
_x$ denote the germs of $f$ and $X$ at a point $x$ of $M$, the tangent vector assigned to $x$ by $v$ 
and the stalk of $\mathcal M$ at $x$, respectively. 
For a compact semialgebraic subset $X$ of a Nash manifold $M$, let $\mathcal N(X)$ denote the germs of Nash 
functions on $X$ in $M$ with the topology of the inductive limit space of the topological spaces $N(U)$ 
with the compact-open $C^\infty$ topology where $U$ runs through the family of open semialgebraic neighborhoods 
of $X$ in $M$.

\begin{thm}\label{sheaf} Let $\mathcal M$ be a coherent sheaf of $\mathcal O$-modules on an analytic manifold $M$. 
\begin{enumerate}
\item\label{A} (Cartan Theorem A) For any $x\in M$ we have $\mathcal M_x=H^0(M,\mathcal M)\mathcal O_x.$
\item\label{Abis} Assume moreover that $\mathcal M_x$ is generated by a uniform number of elements for any $x\in M$. 
Then $H^0(M,\mathcal M)$ is finitely generated as an $H^0(M,\mathcal O)$-module. 
\item\label{B} (Cartan Theorem B) $H^1(M,\mathcal M)=0$.
\item\label{Bbis} Let $X\subset M$ be a global analytic set---the zero set of an analytic 
function. 
Let $\mathcal I$ be a coherent sheaf of $\mathcal O$-ideals on $M$ such that any element of $\mathcal I$ vanishes on 
$X$. 
Then any $f\in H^0(M,\mathcal O/\mathcal I)$ can be extended to some $F\in C^\omega(M)$, i.e. $f$ is the image 
of $F$ under the natural map $H^0(M,\mathcal O)\to H^0(M,\mathcal O/\mathcal I)$. 
If $X$ is normal crossing, we 
can choose $\mathcal I$ to be the function germs vanishing on $X$. 
Then $H^0(M,\mathcal O/\mathcal I)$ consists of functions on $X$ whose germs at each point of $X$ are extensible 
to analytic function germs on $M$. 
\item\label{Oka} (Oka Theorem)
Let $\mathcal M_1$ and $\mathcal M_2$ be coherent sheaves of $\mathcal O$-modules on $M$, and 
$h:\mathcal M_1\to\mathcal M_2$ be an $\mathcal O$-homomorphism. 
Then $\Ker h$ is a coherent sheaf of $\mathcal O$-modules. 
\end{enumerate}
\end{thm}


\subsection{Euclidean realization of a sequence of blowings-up}\label{sect-eu}

Let $C$ be a smooth analytic subset of an analytic manifold $U$, and let $\pi:M \longrightarrow U$ denote 
the blowing-up of $U$ along center $C$. 
In this section, we describe $M$ as a smooth analytic subset of $U\times\PP(k)$ for some $k\in\N$.

Let $\mathcal I$ denote the sheaf of $\mathcal O$-ideals defined by $C$. 
Since $C$ is smooth, each stalk $\mathcal I_x$ is generated by $c =\codim C$ elements. 
Hence there exist a finite number of global generators $h_0,\ldots,h_k \in H^0(U,\mathcal I)$ of $\mathcal I$ by theorem \ref{sheaf}.(\ref{Abis}). 
Define $\mathcal A$ to be the sheaf of relations of $h_0,\ldots,h_k$:
$$
\mathcal A=\cup_{x\in U}\{(\mu_0,\ldots,\mu_k)\in \mathcal O_x^{k+1}: \sum_{i=0}^k \mu_i h_{ix}=0\}. 
$$
Then $\mathcal A$ is coherent by theorem \ref{sheaf}.(\ref{Oka}), and each $\mathcal A_x$ is generated by $k-c+1+(c-1)!$ elements as 
follows. 
If $x_0\not\in C$ then $h_i(x_0)\not=0$ for some $i$, say 0. 
On a small neighborhood of $x_0$, the map $\mathcal O^k\supset\mathcal O^k_x\ni(\mu_1,....,\mu_k)\to(-\sum_{i=1}^
k\mu_ih_{ix}/h_{0x},\mu_1,...,\mu_k)\in\mathcal O_x^{k+1}\subset\mathcal O^{k+1}$ is an isomorphism onto $\mathcal A$. 
Hence $\mathcal A_x$ is generated by $k$ elements. 
If $x_0\in C$, let $x$ denote a point near $x_0$. 
In this case we can assume that $h_{0x},...,h_{c-1x}$ are regular function germs and generate $\mathcal I_x$. 
Then each $h_{ix},\,c\le i\le k$, is of the form $\sum_{i=0}^{c-1}\phi_ih_{ix}$ for some $\phi_i\in
\mathcal O_x$. 
Hence the projection image of $\mathcal A_x$ to the last $k-c+1$ factors of $\mathcal O_x^{k+1}$ is $\mathcal O_x^{k-c
+1}$, and it suffices to see that $\mathcal A_x\cap\mathcal O_x^{c}\times\{0\}\times\cdots\times\{0\}$ is 
generated by $(c-1)!$ elements. 
We do this as follows: $\mathcal A_x\cap\mathcal O_x^{c}\times\{0\}\times\cdots\times\{0\}$ is generated by $(0,...,0,
\overbrace{h_{j-1x}}^i,0,....,\overbrace{-h_{i-1x}}^j,0,...,0)$ for $1\le i<j\le c$. 
Therefore, $\mathcal A$ is generated by its global cross-sections $g_1=(g_{1,0},...,g_{1,k}),...,g_{k'}=(g_{k'
,0},...,g_{k',k})\in C^\omega(U)^{k+1}$ for some $k'\in\N$. 
Moreover, it follows from these arguments that 
\begin{enumerate}
\item[(1)] $\sum_{j=0}^kg_{i,j}h_j=0,\ i=1,...,k'$, 
\item[(2)] 
for each $x\in U-C$, the vectors $g_1(x),...,g_{k'}(x)$ in $\R^{k+1}$ span a hyperplane and $(h_0(x),...,h_k(x))$ 
in $\R^{k+1}$ is non-zero and orthogonal to the hyperplane,
\item[(3)] for each $x\in C$, the linear 
subspace $\{(s_0,...,s_k)\in\R^{k+1}:\sum_{j=0}^ks_jg_{i,j}(x)=0,\ i=1,...,k'\}$ of $\R^{k+1}$ is of 
dimension $c$. 
\end{enumerate}
Hence we can regard set-theoretically $M-\pi^{-1}(C)$ as 
$$
\{(x,t)\in(U-C)\times\PP(k):t_ih_j(x)=t_jh_i(x),\,i,j=0,\ldots,k\} 
$$
by (2), hence $M$ as
$$
\{(x,t)\in U\times\PP(k):t_ih_j(x)=t_jh_i(x),\,i,j=0,\ldots,k,\ \text{and}\\
\sum_{j=0}^kt_jg_{i,j}(x)=0,\,i=1,...,k'\} 
$$
by (3) and by $\sum_{j=0}^kt_jg_{i,j}(x)=0,\,i=1,...,k'$, for $(x,t)\in(U-C)\times\PP(k)$ with $t_ih
_j(x)=t_jh_i(x),\,i,j=0...,k$, and $\pi$ as the restriction to $M$ of the projection $U\times\PP(k)\to U$. 
When we identify $M$ with the subset of $U\times\PP(k)$, we say $M$ is {\it realized} in $U\times\PP(k)$.

Since we treat only finite sequences of blowings-up, we can embed $M$ into a Euclidean space. 
For that we embed $\PP(k)$ algebraically in $\R^{(k+1)^2}$ as in \cite{BCR} by 
$$
(t_0:\ldots:t_k)\mapsto(\frac{t_it_j}{|t|^2}), 
$$
where $|t|^2=\sum_{i=0}^k t_i^2$. 
It is known that $\PP(k)$ is a non-singular algebraic subvariety in $\R^{(k+1)^2}$. 
We denote by $y_{i,j}$ the coordinates on $\R^{(k+1)^2}$ such that $y_{i,j}=t_it_j/|t|^2$ on $\PP(k)$. 
Let $\xi_1,\ldots,\xi_s$ be generators of the ideal of $\R[y_{i,j}]$ of functions vanishing on $\PP(k)$. 
Set $l_{i,j,m}(x,y)=y_{i,j}h_m(x)-y_{m,i}h_j(x)$ for $ i,j,m=0,\ldots,k$. 
Define
$$
N=\{(x,y)\in U\times\R^{(k+1)^2}:l_{i,j,m}(x,y)=0,\ i,j,m=0,\ldots,k,$$
$$\sum_{j=0}^ky_{j,m}g_{i,j}(x)=0,\,i=1,...,k',\,m=0,...,k,\ \text{and}\ \xi_i(y)=0,\,i=1,...,s\}.
$$
Then $M=N$. Moreover the analytic sets on both sides coincide algebraically, i.e. the functions $l_{i,j,m},\,\sum
_{j=0}^ky_{j,m}g_{i,j},\,\xi_i$ generate $I(M)$---the ideal of $C^\omega(U\times\R^{(k+1)^2})$ 
of functions vanishing on $M$. Indeed, by theorem \ref{sheaf}.(\ref{Bbis}) the problem is local. 
If $x\in U-C$, the claim locally at $x$ is clear. 
Assume that $x\in C$, and let $(x_1,...,x_n)$ denote a local coordinate system of $U$ 
around $x$. 
As the claim does not depend on the choice of $\{g_i\}$, we can assume that $h_j=x_{j+1},\,j=0,...,c-1,\ h_j=
\sum_{i=0}^{c-1}\phi_{i,j}h_i,\,j=c,...,k,$ and $ g_1=(-\phi_{0,c},...,-\phi_{c-1,c},1,0,...,0),...,g_{k-c+1}
=(-\phi_{0,k},...,-\phi_{c-1,k},0,...,0,1),\ g_{k-c+2}\!=\!(x_2,-x_1,0,\!....,\!0),...,g_{k'}\!=\!(0,\!...,\!0,
x_{c},\!-x_{c-1},0,...,0)$ for some $C^\omega$ functions $\phi_{i,j}$ on a neighborhood of $x$ 
and for $k'=k-c+1+(c-1)!$. 
Then 

\begin{displaymath}
(t_0\cdots t_k)=(t_0\cdots t_{c-1})
\left( \begin{array}{cccccc}
1&&0&\phi_{0,c}&\cdots&\phi_{0,k}\\
&\ddots&&\vdots&&\vdots\\
0&&1&\phi_{c-1,c}&\cdots&\phi_{c-1,k}\\
\end{array} \right),
\end{displaymath}

\begin{displaymath}
\left( \begin{array}{c}
t_0\\
\vdots\\
t_k
\end{array} \right)
=
\left( \begin{array}{ccc}
1&&0\\
&\ddots&\\
0&&1\\
\phi_{0,c}&\cdots&\phi_{c-1,c}\\
\vdots&&\vdots\\
\phi_{0,k}&\cdots&\phi_{c-1,k}
\end{array} \right)
\left( \begin{array}{c}
t_0\\
\vdots\\
t_{c-1}
\end{array} \right).
\end{displaymath}

Therefore the matrix $(t_i,t_j)_{i,j=0,\ldots,k}$ is equal to

\begin{displaymath}
\left( \begin{array}{ccc}
1&&0\\
&\ddots&\\
0&&1\\
\phi_{0,c}&\cdots&\phi_{c-1,c}\\
\vdots&&\vdots\\
\phi_{0,k}&\cdots&\phi_{c-1,k}
\end{array} \right)
\left( \begin{array}{c}
t_0\\
\vdots\\
t_{c-1}
\end{array} \right)
(t_0\cdots t_{c-1})
\left( \begin{array}{cccccc}
1\!&&0&\phi_{0,c}&\cdots&\phi_{0,k}\\
\!&\ddots&&\vdots&&\vdots\\
0\!&&1&\phi_{c-1,c}&\cdots&\phi_{c-1,k}
\end{array} \right)
\end{displaymath}

whereas the matrix ${(y_{i,j})}_{i,j=0,\ldots,k}$ equals

\begin{displaymath}
\left( \begin{array}{ccc}
1\!&\!&\!0\\
\!&\!\ddots\!&\!\\
0\!&\!&\!1\\
\phi_{0,c}\!&\!\cdots\!&\!\phi_{c-1,c}\\
\vdots\!&\!&\!\vdots\\
\phi_{0,k}\!&\!\cdots\!&\!\phi_{c-1,k}
\end{array} \right)
\left( \begin{array}{ccc}
y_{0,0}\!&\!\cdots\!&\!y_{0,c-1}\\
\vdots&&\vdots\\
y_{c-1,0}\!&\!\cdots\!&\!y_{c-1,c-1}
\end{array} \right)
\left( \begin{array}{cccccc}
1\!&\!&\!0\!&\!\phi_{0,c}\!&\!\cdots\!&\!\phi_{0,k}\\
\!&\!\ddots\!&\!&\!\vdots\!&\!&\!\vdots\\
0\!&\!&\!1\!&\!\phi_{c-1,c}\!&\!\cdots\!&\!\phi_{c-1,k}
\end{array} \right).
\end{displaymath}

Hence we can forget $h_j$ and $y_{i,j}=y_{j,i},\ i=0,...,k,\,j=c,...,k$, and we can replace $N$ with its image under the projection $U\times\R^{(k+1)^2}
\ni(x,y_{i,j})\to(x,y_{i,j})_{i,j\le c-1}\in U\times\R^{c^2}$. 
Considering the realization of $M$ in $U\times\PP(k)$, it becomes 
$$
\widetilde M=\{(x,y)\in U\times\R^{c^2}:l_{i,j,m}(x,y)=0,\ i,j,m=0,\ldots,c-1,$$
$$\sum_{j=0}^{c-1}y_{j,m}g_{i,j}(x)\!=\!0,\,i=k-c+2,...,k',m=0,...,c-1,\ \text{and}\ \xi'_i(y)=0,\,i=1,...,s'\},
$$
where $\xi'_i$ are generators of $I(\PP(c-1))\subset\R[y_{i,j}]_{i,j\le c-1}$. Therefore it suffices to 
show that $l_{i,j,m},\,\sum_{j'=0}^{c-1}y_{j',m}g_{i',j'},\,\xi'_{i''},\,i,j,m=0,...,c-1,\,i'=k-c+2,...,k',
\,i''=1,...,s'$, generate $I(\widetilde M)$. 
However, by easy calculations we prove that $l_{i,j,m}$ and $\xi'_{i''}$ generate $I(\widetilde M)$. 
(To realize $M$ in $U\times\PP(k)$ we need the equations $\sum_{j=0}^kt_jg_{i,j}(x)=0,\ i=1,...,k-c+1$, 
which are equivalent to $t_c=t_0\phi_{0,c}+\cdots+t_{c-1}\phi_{c-1,c},...,t_k=t_0\phi_{0,k}+\cdots t_{c-1}
\phi_{c-1,k}$.) 

\subsection{Perturbation of a blowing-up}

When we perturb $h_i,\ i=0,...,k$, in the strong Whitney $C^\infty$ topology, the common zero set $Z(h_i)$ of $h_i$'s may become of smaller 
dimension than $C$ and singular, where the strong Whitney $C^\infty$ topology on $C^\infty(U)$ is defined to be the topology of the projective limit space of the topological spaces $C^\infty(U_k)$ with the $C^\infty$ topology for all compact $C^\infty$ submanifolds possibly with boundary $U_k$ of $U$. (Note that Whitney Approximation Theorem in \cite{W} holds also in this topology, and we call it Whitney Approximation Theorem.) 
However, we have

\begin{lemma}\label{pert}
Let $\tilde h_i,\ i=0,...,k$, and $\tilde g_i=(\tilde g_{i,0},...,\tilde g_{i,k}),\ i=1,...,k'$, be $C^\omega$ 
functions on $U$ and $C^\omega$ maps from $U$ to $\R^{k+1}$ close to $h_i$ and $g_i$, respectively, in the 
strong Whitney $C^\infty$ topology. 
Assume that (1) $\sum_{j=0}^k\tilde g_{i,j}\tilde h_j=0,\ i=1,...,k'$. 
Then 
\begin{itemize}
\item $\tilde C=Z(\tilde h_i)$ is smooth and of the same dimension as $C$, $\tilde h_0,...,\tilde h_k$ 
generate $I(Z(\tilde h_i))$ and $\tilde g_1,...,\tilde g_{k'}$ are generators of the sheaf of relations 
$\tilde {\mathcal A}$ of $\tilde h_0,...,\tilde h_k$. 
\item Let $\pi:M\to U$ and $\tilde\pi:\tilde M\to U$ denote the blowings-up along centers $C$ and $\tilde C$, 
respectively. 
Let $M$ and $\tilde M$ be realized in $U\times\PP(k)$ as in section \ref{sect-eu}. 
Then there exist analytic diffeomorphisms $\tau$ of $U$ and $\psi:M\to\tilde M$ close to id in the strong 
Whitney $C^\infty$ topology such that $\tau(C)=\tilde C$ and $\tilde\pi\circ\psi=\tau\circ\pi$.
\end{itemize} 
\end{lemma}

\begin{demo}
The problem in the former half is local and clear around a point outside of $C$, and hence we 
assume that $h_j=x_{j+1},\,j=0,...,c-1$, for a local coordinate system $(x_1,...,x_n)$, and $h_j=\sum_{i=0}^
{c-1}\phi_{i,j}h_i,\,j=c,...,k,$ for some $C^\omega$ functions $\phi_{i,j}$ on $U$. 
Then $Z(\tilde h_0,...,\tilde h_{c-1})$ is smooth and of the same dimension as $C$. 
Hence we need to see that $\tilde h_j,\ j=c,...,k$, are contained in the ideal of $C^\omega(U)$ generated 
by $\tilde h_j,\ j=0,...,c-1$. 
Choose $C^\omega$ functions $\alpha_{i,j},\ i=1,...,k-c+1,\,j=1,...,k'$, on $U$ so that 
$$\begin{pmatrix}
\alpha_{1,1}&\cdots&\alpha_{1,k'}\\
\vdots&&\vdots\\
\alpha_{k-c+1,1}&\cdots&\alpha_{k-c+1,k'}
\end{pmatrix}
\begin{pmatrix}
g_1\\
\vdots\\
g_{k'}
\end{pmatrix}$$
is of the form 
$$\begin{pmatrix}
-\phi_{0,1}&\cdots&-\phi_{c-1,1}&1&&0\\
\vdots&&\vdots&&\ddots&\\
-\phi_{0,k-c+1}&\cdots&-\phi_{c-1,k-c+1}&0&&1
\end{pmatrix}.$$
Set
$$
\begin{pmatrix}
\tilde g'_{1,0}&\cdots&\tilde g'_{1,k}\\
\vdots&&\vdots\\
\tilde g'_{k-c+1,0}&\cdots&\tilde g'_{k-c+1,k}
\end{pmatrix}=\begin{pmatrix}
\alpha_{1,1}&\cdots&\alpha_{1,k'}\\
\vdots&&\vdots\\
\alpha_{k-c+1,1}&\cdots&\alpha_{k-c+1,k'}
\end{pmatrix}
\begin{pmatrix}
\tilde g_1\\
\vdots\\
\tilde g_{k'}
\end{pmatrix}.
$$
Then 
$\begin{pmatrix}
\tilde g'_{1,c}&\cdots&\tilde g'_{1,k}\\
\vdots&&\vdots\\
\tilde g'_{k-c+1,c}&\cdots&\tilde g'_{k-c+1,k}
\end{pmatrix}$ 
is close to 
$\begin{pmatrix}
1&&0\\
&\ddots&\\
0&&1
\end{pmatrix}
$. 
Hence 
$$
\begin{pmatrix}
\tilde g'_{1,c}&\cdots&\tilde g'_{1,k}\\
\vdots&&\vdots\\
\tilde g'_{k-c+1,c}&\cdots&\tilde g'_{k-c+1,k}
\end{pmatrix}^{-1}
\begin{pmatrix}
\alpha_{1,1}&\cdots&\alpha_{1,k'}\\
\vdots&&\vdots\\
\alpha_{k-c+1,1}&\cdots&\alpha_{k-c+1,k'}
\end{pmatrix}
\begin{pmatrix}
\tilde g_1\\
\vdots\\
\tilde g_{k'}
\end{pmatrix}
$$
is well-defined and of the form 
$\begin{pmatrix}
\tilde g''_{1,0}&\cdots&\tilde g''_{1,c-1}&1&&0\\
\vdots&&\vdots&&\ddots&\\
\tilde g''_{k-c+1,0}&\cdots&\tilde g''_{k-c+1,c-1}&0&&1
\end{pmatrix}$.
Now (1) implies 
$\begin{pmatrix}
\tilde g_1\\
\vdots\\
\tilde g_{k'}
\end{pmatrix}
\begin{pmatrix}
\tilde h_0\\
\vdots\\
\tilde h_k
\end{pmatrix}
=\begin{pmatrix}
0\\
\vdots\\
0
\end{pmatrix}$. 
Therefore, 
$$
\begin{pmatrix}
\tilde g''_{1,0}&\cdots&\tilde g''_{1,c-1}&1&&0\\
\vdots&&\vdots&&\ddots&\\
\tilde g''_{k-c+1,0}&\cdots&\tilde g''_{k-c+1,c-1}&0&&1
\end{pmatrix}
\begin{pmatrix}
\tilde h_0\\
\vdots\\
\tilde h_k
\end{pmatrix}=
\begin{pmatrix}
0\\
\vdots\\
0
\end{pmatrix}
$$
and $\tilde h_j=-\sum_{i=0}^{c-1}\tilde g''_{j-c+1,i}\tilde h_i,\ j=c,...,k$.

We need to see that $\tilde g_1,...,\tilde g_{k'}$ are generators of $\tilde {\mathcal A}$. 
By (1) they are global cross-sections of $\tilde {\mathcal A}$. 
We postpone proving $\tilde g_1,...,\tilde g_{k'}$ generate $\tilde{\mathcal A}$.

Next we prove the latter half of the lemma. 
We first define $\tau$ on $C$. 
The condition on $\tau|_C$ to be satisfied is $\tau(C)=\tilde C$. 
Let $U\subset\R^N$, let $q$ denote the orthogonal projection of a tubular neighborhood of $U$ in $\R^N$, and 
let $p:V\to C$ denote the proper orthogonal projection of a small closed tubular neighborhood of $C$ in $U$. 
We want to require $\tau$ to satisfy, moreover, $p\circ\tau=\id$ on $C$. 
Then $\tau|_C$ is unique and the problem of finding $\tau|_C$ is local. 
Hence we assume as above that $h_j=x_{j+1},\,j=0,...,c-1$, for a local coordinate system $(x_1,...,x_n)$ at each 
point of $C$ and $h_j=\sum_{i=0}^{c-1}\phi_{i,j}h_i,\,j=c,...,k,$ for some $C^\omega$ functions $\phi_{i,j}$. 
Then $\tau|_C$ is well-defined (cf. proof of lemma 4.2 in \cite{FS}), and $\tau|_C$ is an analytic embedding of $C$ 
into $U$ close to id in the strong Whitney $C^\infty$ topology.

Secondly, we extend $\tau|_C$ to $V$ by setting $\tau(x)=q(\tau\circ p(x)+x-p(x))$ for $x\in V$, which is 
close to id in the strong Whitney $C^\infty$ topology. 
Moreover, using the extension, we extend $\tau|_C$ to an analytic diffeomorphism $\tau$ of $U$ close to id in 
the same topology by using a partition of unity of class $C^\infty$
and theorem \ref{sheaf}.(\ref{Bbis}), combined with Whitney Approximation 
Theorem \cite{W}.

Lastly, we need to find $\psi$. 
Set $\Tilde{\Tilde h}_i=h_i\circ\tau^{-1}$ and $\Tilde{\Tilde g}_i=g_i\circ\tau^{-1}$ and define $\Tilde{
\Tilde M}$ by $\Tilde{\Tilde h}_i$ and $\Tilde{\Tilde g}_i$ in $U\times\PP(k)$ and $\Tilde{\Tilde\pi}:\Tilde{
\Tilde M}\to U$. 
Then 
$$
\Tilde{\Tilde M}=\{(x,t)\in U\times\PP(k):t_ih_j\circ\tau^{-1}(x)=t_jh_i\circ\tau^{-1}(x),\,i,j=0,\ldots,k,\ 
\text{and}$$
$$
\sum_{j=0}^kt_jg_{i,j}\circ\tau^{-1}(x)=0,\,i=1,...,k'\}, 
$$
and $\Tilde{\Tilde\psi}:M\ni(x,t)\to(\tau(x),t)\in\Tilde{\Tilde M}$ is an analytic diffeomorphism close to 
id in the strong Whitney $C^\infty$ topology such that $\Tilde{\Tilde\pi}\circ\Tilde{\Tilde\psi}=\tau\circ\pi$. 
Hence we can replace $h_i$ and $g_i$ with $h_i\circ\tau^{-1}$ and $g_i\circ\tau^{-1}$, respectively. 
Thus we assume from the beginning that $Z(h_i)=Z(\tilde h_i)$. 
Set $h'_i=\tilde h_i-h_i$. 
Then there exist analytic functions $\chi_{i,j}$, $i,j=0,...,k$, on $U$ close to 0 in the topology such that 
$h'_i=\sum_{j=0}^k\chi_{i,j}h_j$, $i=0,...,k$, which is proved, as
before, by using a partition of unity of class $C^\infty$ and theorem
\ref{sheaf}.(\ref{Bbis}) combined with Whitney Approximation Theorem.

Set $A(x)=\begin{pmatrix}
1&&0\\
&\ddots&\\
0&&1
\end{pmatrix}
+
\begin{pmatrix}
\chi_{0,0}&\cdots&\chi_{k,0}\\
\vdots&&\vdots\\
\chi_{0,k}&\cdots&\chi_{k,k}
\end{pmatrix}$ and define an analytic diffeomorphism $\psi$ of $U\times\PP(k)$ by $\psi(x,t)=(x,tA(x))$ for 
$(x,t)\in U\times\PP(k)$. 
Then $(h_0,...,h_k)A=(\tilde h_0,...,\tilde h_k)$ on $U$, hence $\psi(M)=\tilde M$, $\tilde\pi\circ\psi=
\pi$ on $M$ and $\psi|_M$ is close to id in the topology, which proves the latter half.

It remains to show that $\tilde g_1,...,\tilde g_{k'}$ generate $\tilde {\mathcal A}$, i.e. $\tilde g_1,...,\tilde g
_{k'}$ generate the $C^\omega(U)$-module $\tilde X$ defined by $\tilde X=\{\tilde g\in(C^\omega(U))^{k+1}
:\tilde g\begin{pmatrix}\tilde h_0\\\vdots\\\tilde h_k\end{pmatrix}=0\}$ by theorem \ref{sheaf}.(\ref{B}).  
As above we can assume that $Z(h_i)=Z(\tilde h_i)$. 
Moreover, we suppose that $\tilde h_i=h_i$ for any $i$ for the following reason. 
For the above $A$ we have ${}^tA\begin{pmatrix} h_0\\\vdots\\h_k\end{pmatrix}=\begin{pmatrix}\tilde h_0\\\vdots\\\tilde h_k
\end{pmatrix}$. 
Hence $\tilde X=\{\tilde g\in(C^\omega(U))^{k+1}:\tilde g{}^tA\begin{pmatrix} h_0\\\vdots\\ h_k\end{pmatrix}=0\}$. 
Then it suffices to see that $\tilde g_1{}^tA,...,\tilde g_{k'}{}^tA$ generate the $C^\omega(U)$-module $X=\{g\in(C^\omega(U))^{k+1}:g\begin{pmatrix} h_0\\\vdots\\ h_k\end{pmatrix}=0\}$ because the map 
$(C^\omega(U))^{k+1}\ni\tilde g\to\tilde g{}^tA\in(C^\omega(U))^{k+1}$ is an isomorphism as $C^\omega(U)
$-modules. 
Here $\tilde g_1{}^tA,...,\tilde g_{k'}{}^tA$ are close to $g_1,...,g_{k'}$ respectively. 
Therefore, replacing $\begin{pmatrix}\tilde h_0\\\vdots\\\tilde h_k\end{pmatrix}$ and $\begin{pmatrix}\tilde g_0\\\vdots\\
\tilde g_{k'}\end{pmatrix}$ with $\begin{pmatrix} h_0\\\vdots\\ h_k\end{pmatrix}$ and $\begin{pmatrix}\tilde g_0\\\vdots\\
\tilde g_{k'}\end{pmatrix}{}^tA$, respectively, we suppose from the beginning that $\tilde h_i=h_i$ for all $i$ 
and $\tilde{\mathcal A}=\mathcal A$.

As above, the problem is local at each point of $C$ and we assume that $h_j=x_{j+1},\ j=0,...,c-1$, for a local 
coordinate system $(x_1,...,x_n)$. 
Recall that
\begin{equation}\tag{*}
\begin{pmatrix}
\beta_{1,1}&\cdots&\beta_{1,k'}\\
\vdots&&\vdots\\
\beta_{k-c+1,1}&\cdots&\beta_{k-c+1,k'}
\end{pmatrix}
\begin{pmatrix}
\tilde g_1\\
\vdots\\
\tilde g_{k'}
\end{pmatrix}
=\begin{pmatrix}
\cdot&\cdots&\cdot&1&&0\\
\vdots&&\vdots&&\ddots&\\
\cdot&\cdots&\cdot&0&&1
\end{pmatrix}
\end{equation} 
for some $C^\omega$ functions $\beta_{i,j}$ on $U$. 
Let $r$ be the restriction to $\mathcal A$ of the projection of $\mathcal O^{k+1}$ to the last $k-c+1$ factors 
and let $r_*:H^0(U,\mathcal A)\to(C^\omega(U))^{k-c+1}$ denote the induced map. 
Then $(*)$ implies that $r_*(\tilde g_1),...,r_*(\tilde g_{k'})$ generate $r(\mathcal A)=\mathcal O^{k-c+1}$. 
Hence it suffices to see that $\mathcal A\cap\mathcal O^c\times\{0\}\times\cdots\times\{0\}$ is generated by $\{\sum_{i=
1}^{k'}\beta_i\tilde g_i:\beta_i\in C^\omega(U),\ r_*(\sum_{i=1}^{k'}\beta_i\tilde g_i)=0\}$.

Since $g_1,...,g_{k'}$ generate $\mathcal A$, there exist $C^\omega$ functions $\gamma_{i,j},\ i=1,...,(c-1)!,
\ j=1,...,k'$, such that 
$$
\begin{pmatrix}\gamma_{1,1}&\cdots&\gamma_{1,k'}\\
\vdots&&\vdots\\
\gamma_{(c-1)!,1}&\cdots&\gamma_{(c-1)!,k'}\end{pmatrix}\begin{pmatrix} g_1\\
\vdots\\g_{k'}
\end{pmatrix}
=\begin{pmatrix} x_2\!&\!-x_10\cdots\!&\!0\!&\!0&\cdots&0\\
\vdots&&\vdots&\vdots&&\vdots\\
0\!&\cdots0x_c\!&\!-x_{c-1}\!&\!0&\cdots&0
\end{pmatrix},
$$
whose rows are global generators of $\mathcal A\cap\mathcal O^c\times\{0\}\times\cdots\times\{0\}$. 
Consider the matrix 
$\begin{pmatrix}\gamma_{1,1}&\cdots&\gamma_{1,k'}\\\vdots&&\vdots\\\gamma_{(c-1)!,1}&\cdots&\gamma_{(c-1)!,
k'}\end{pmatrix}\begin{pmatrix}\tilde g_1\\\vdots\\\tilde g_{k'}\end{pmatrix}$. 
Its $(i,j)$ components, $i=1,...,(c-1)!,\ j=(c-1)!+1,...,k'$, are close to 0. 
Hence by $(*)$ there exist $C^\omega$ functions $\delta_{i,j},\ i=1,...,(c-1)!,\ j=1,...,k'$, close to 0 such 
that the $(i,j)$ components, $i=1,...,(c-1)!,\ j=(c-1)!+1,...,k'$, of the matrix 
$$\begin{pmatrix}\gamma_{1,1}&\cdots&
\gamma_{1,k'}\\\vdots&&\vdots\\\gamma_{(c-1)!,1}&\cdots&\gamma_{(c-1)!,k'}\end{pmatrix}\begin{pmatrix}\tilde g_1\\\vdots
\\\tilde g_{k'}\end{pmatrix}-\begin{pmatrix}\delta_{1,1}&\cdots&\delta_{1,k'}\\\vdots&&\vdots\\\delta_{(c-1)!,1}&\cdots
&\delta_{(c-1)!,k'}\end{pmatrix}\begin{pmatrix}\tilde g_1\\\vdots\\\tilde g_{k'}\end{pmatrix}$$ are 0. 
Replace $\gamma_{i,j}$ with $\tilde\gamma_{i,j}=\gamma_{i,j}-\delta_{i,j}$. 
Then the $(i,j)$ components, $i=1,...,(c-1)!,\ j=(c-1)!+1,...,k'$, of the matrix $\begin{pmatrix}\tilde\gamma_{1,1}&
\cdots&\tilde\gamma_{1,k'}\\\vdots&&\vdots\\\tilde\gamma_{(c-1)!,1}&\cdots&\tilde\gamma_{(c-1)!,k'}\end{pmatrix}
\begin{pmatrix}\tilde g_1\\\vdots\\\tilde g_{k'}\end{pmatrix}$ are 0, and each row is an approximation of the 
corresponding row of the matrix $\begin{pmatrix} x_2\!&\!-x_10\cdots\!&\!0\!&\!0&\cdots&0\\\vdots&&\vdots&\vdots&&
\vdots\\0\!&\cdots0x_c\!&\!-x_{c-1}\!&\!0&\cdots&0\end{pmatrix}$. 
Therefore, we can suppose from the beginning that $k=c-1,\ k'=(c-1)!,\,g_1=(x_2,x_1,0,...,0),...,g_{k'}=(0,...,0,
x_c,-x_{c-1})$.

Let $\mathfrak m_x$ denote the maximal ideal of $\mathcal O_x$ for $x\in U$. 
For each $x\in C$, $\tilde g_{1x},...,\tilde g_{k'x}$ generate $\mathcal A_x$ if and only if $\tilde g_{1x},...,
\tilde g_{k'x}$ and $\mathfrak m_x\mathcal A_x$ generate $\mathcal A_x$ by Nakayama lemma. 
On the other hand, the images of $g_{1x},...,g_{k'x}$ in the linear space $\mathcal A_x/\mathfrak m_x\mathcal A_x$, $x\in 
C$, are a basis and hence $\mathcal A_x/\mathfrak m_x\mathcal A_x$ is a linear space of dimension $k'$. 
Hence it suffices to see that the images of $\tilde g_{1x},...,\tilde g_{k'x}$ in $\mathcal A_x/\mathfrak m_x\mathcal A_x$, $x
\in C$, are linearly independent. 
Here $\mathfrak m_x\mathcal A_x=\mathfrak m_x^2\mathcal O_x^c\cap\mathcal A_x$ because clearly $\mathfrak m_x\mathcal A_x\subset\mathfrak m_x
^2\mathcal O_x^c\cap\mathcal A_x$ and $\mathcal A_x/(\mathfrak m_x^2\cap\mathcal A_x)\,(\approx(\mathfrak m_x^2\mathcal O_x^c+\mathcal A_x)/
\mathfrak m_x^2\mathcal O_x^c)$ and $\mathcal A_x/\mathfrak m_x\mathcal A_x$ are linear spaces of the same dimension. 
Now $\cup_{x\in U}\mathcal O_x^c/\mathfrak m_x^2\mathcal O_x^c$ coincides with the space of 1-jets from $U$ to $\R^c$. 
Hence for $x\in C$, the images of $\tilde g_{1x},...,\tilde g_{k'x}$ in $\mathcal O_x^c/\mathfrak m_x^2\mathcal O_x^c$, i.e. 
in $\mathcal A_x/(\mathfrak m_x^2\mathcal O_x^c\cap\mathcal A_x)$ are linearly independent because $\tilde g_1,...,\tilde g_{k'}$ 
are sufficiently close to $g_1,...,g_{k'}$, respectively, in the Whitney $C^1$ topology and the images of 
$g_{1x},...,g_{k'x}$ are linearly independent. 

\end{demo}

\begin{rmk}\label{rm-lem}\begin{flushleft}\end{flushleft}
\begin{enumerate}
\item In lemma \ref{pert}, $\tau|_C$ is an embedding of $C$ into $U$ close to id in the strong Whitney $C^\infty$ 
topology such that $\tau(C)=\tilde C$. 
Conversely, assume that there exist an analytic embedding $\tau_C$ of $C$ into $U$ close to id in the same topology. 
Then $\tau_C$ is extensible to an analytic diffeomorphism $\tau$ of $U$ close to id in the topology. 
Define $\tilde C,\ \tilde h_i\ \tilde g_i$ and $\tilde\pi:\tilde M\to\tilde U$ to be $\tau(C),\ h_i\circ\tau
^{-1},\ g_i\circ\tau^{-1}$ and the blowing-up of $U$ along center $\tilde C$, respectively. 
Realize $M$ and $\tilde M$ in $U\times\PP(k)$ as before. 
Then $\tilde h_i$ and $\tilde g_i$ are close to $h_i$ and $g_i$ respectively, $\sum_{j=0}^
k\tilde g_{i,j}\tilde h_j=0$, and hence by lemma \ref{pert} there exists an analytic diffeomorphism $\psi:M\to
\tilde M$ close to id in the topology such that $\tilde\pi\circ\psi=\tau\circ\pi$.

When there exists this kind of $\tau_C$, we say $\tilde C$ is {\it close} to $C$ in the strong Whitney 
$C^\infty$ topology. 
Let $\psi:M_1\to M_2$ and $\tilde\psi:\tilde M_1\to\tilde M_2$ be analytic maps between analytic manifolds 
with $M_1\subset\R^{n_1},\ \tilde M_1\subset\R^{n_1},\ M_2\subset\R^{n_2}$ and $\tilde M_2\subset\R^{n_2}$. 
Assume that $\tilde M_1$ is close to $M_1$ in the topology through an analytic diffeomorphism $\tau:M_1\to
\tilde M_1$ close to id in the topology. 
Then we say $\tilde\psi$ is {\it close} to $\psi$ in the topology if
$\tilde\psi \circ\tau$ is so to $\psi$. 
\item \label{r2} The germ case of lemma \ref{pert} holds in the following sense. 
Let $h_i,\ g_i,\ U$ and $C$ be the same as above. 
Let $X$ be a compact subset of $U$, and let $\tilde h_i$ and $\tilde g_i$ be $C^\omega$ functions and maps 
defined on an open neighborhood $V$ of $X$ in $U$ close to $h_i|_V$ and $g_i|_V$, respectively, in the 
compact-open $C^\infty$ topology with $\sum_{j=0}^k\tilde g_{i,j}\tilde h_j=0$. 
Shrink $V$. 
Then the same statement as the former half of lemma \ref{pert} holds. 
For the latter half, let $\pi:M\to U$ and $\tilde\pi:\tilde M\to V$ denote the blowings-up along centers 
$C$ and $\tilde C=Z(\tilde h_i)$. 
Let $M\subset U\times\PP(k)$ and $\tilde M\subset V\times\PP(k)$ be as above. 
Then there exist analytic embeddings $\tau_-:V\to U$ and $\psi_-:\tilde M\to M$ close to id in the 
compact-open $C^\infty$ topology such that $\tau_-\circ\tilde\pi=\pi\circ\psi_-$. 
In this case we say $\tilde C$ is {\it close} to $C$ at $X$ in the $C^\infty$ topology, and define {\it closeness} 
of an analytic map to another one at a compact set. 
\end{enumerate}
\end{rmk}

\subsection{Nash approximation}

We state and prove a Nash approximation theorem of an analytic desingularization of a Nash function. This result will be crucial for the proof of theorem \ref{main}.

\begin{thm}\label{Nresol} 
Let $f$ be a Nash function on a Nash manifold $M$. 
Let $M_m \xrightarrow{\pi_m} M_{m-1}\longrightarrow\cdots \xrightarrow{\pi_1} M_0=M$ 
be a sequence of blowings-up along smooth analytic centers $C_{m-1}$ in $M_{m-1},...,C_0$ in $M_0$, 
respectively, such that $f\circ\pi_1\circ\cdots\circ\pi_m$ has only normal crossing singularities.  
Let $X$ be a compact subset of $M$. 
Then there exist an open semialgebraic neighborhood $U$ of $X$ in $M$, a sequence of blowings-up $U_m
\xrightarrow{\tau_m} U_{m-1}\longrightarrow\cdots \xrightarrow{\tau_1} U_0=U$ along 
smooth Nash centers $D_{m-1}$ in $U_{m-1},...,D_0$ in $U_0$, respectively, and an analytic embedding $\psi:
U_m\to M_m$ such that $\psi(\tau_m^{-1}(D_{m-1}))\subset\pi_m^{-1}(C_{m-1}),...,\psi((\tau_1\circ\cdots\circ
\tau_m)^{-1}(D_0))\subset(\pi_1\circ\cdots\circ\pi_m)^{-1}(C_0)$ and $f\circ\pi_1\circ\cdots\circ\pi_m\circ\psi=f
\circ\tau_1\circ\cdots\circ\tau_m$. 
Let $M_1,...,M_m$ be realized in $M\times\PP(k_0),...,M\times\PP(k_0)\times\cdots\times\PP(k_{m-1})$, 
respectively, for some $k_0,...,k_{m-1}\in\N$. 
Then we can realize $U_1,...,U_m$ in $U\times\PP(k_0),...,U\times\PP(k_0)\times\cdots\times\PP(k_{m-1})$, 
respectively, so that each pair $D_i\subset U_i$ and $\psi$ are close to $C_i\subset M_i$ at $(\tau_1\circ
\cdots\circ\tau_{i-1})^{-1}(X)$ and to id at $(\tau_1\circ\cdots\circ\tau_m)^{-1}(X)$, respectively, in 
the $C^\infty$ topology. 
\end{thm}

The proof of theorem \ref{Nresol} is the heart of the paper. It consists in a combination of algebra and topology, via a nested N\'eron  Desingularization Theorem 
(see Theorem 11.4, \cite{Sp}) and Nash Approximation Theorem. 
We proceed as follows. 
First we describe the analytic situation of the sequence of blowings-up in terms of ideals. 
Next we apply the nested version of N\'eron Desingularization Theorem and come down to a regular situation. 
Then, in the regular situation, the classical Nash Approximation Theorem enables to realize the 
approximation. 
The idea comes from the proof of Theorem 1.1 in \cite{CRS}, where the usual N\'eron Desingularization Theorem is 
used.

\vskip 5mm

\begin{demo} Consider the blowing-up $\pi_1:M_1\to M_0$ along center $C_0$. 
Let 
\begin{itemize}
\item $\mathcal I^0$ denote the sheaf of $\mathcal O$-ideals on $M_0$ defined by $C_0$, 
\item $h_0^0,...,h_{k_0}^0$ its 
global generators, 
\item $\mathcal A^0\subset\mathcal O_{M_0}^{k_0+1}$ the sheaf of relations of $h_0^0,...,h_{k_0}^0$, 
\item $g_1^0=(g_{1,0}^0,...,g_{1,k_0}^0),..., g_{k'_0}^0=(g_{k'_0,0}^0,...,g_{k'_0,k_0}^0)$ 
global generators of $\mathcal A^0$, 
\item $\xi_1^0,...,\xi_{s_0}^0$ generators of the ideal $I(\PP(k_0))$ of 
$\R[y_{i,j}^0]_{0\le i,j\le k_0}$ of functions vanishing on $\PP(k_0)$ in $\R^{(k_0+1)^2}$. 
\end{itemize}

Set $l_{i_1,i_2,i_3}^0(x,y^0)=y_{i_1,i_2}^0h_{i_3}^0(x)-y_{i_3.i_1}^0h_{i_2}^0(x)$ for $(x,y^0)\in M_0\times
\R^{(k_0+1)^2}$ and for $i_1,i_2,i_3=0,...,k_0$. 
Then 

\begin{equation}\tag{1}
\sum_{j=0}^{k_0}g_{i,j}^0(x)h_j^0(x)=0\quad\text{on}\ M_0\ \text{for}\ i=1,...,k'_0,
\end{equation}
and $M_1$ is generated by those
$l_{i_1,i_2,i_3}^0,\sum_{j=0}^{k_0}y_{j,i_1}^0g_{i_2,j}^0$ and
$\xi_i^0$, namely
$$M_1=\{(x,y^0)\in M_0\times\R^{(k_0+1)^2}:l_{i_1,i_2,i_3}^0(x,y^0)=0,\ i_1,i_2,i_3=0,\ldots,k_0,$$
$$\sum_{j=0}^{k_0}y_{j,i_1}^0g_{i_2,j}^0(x)=0,\,i_1=0,...,k_0,\,i_2=1,...,k'_0,\,\text{and}\ \xi_i^0(y^0)=0,
\,i=1,...,s_0\}.$$
Let $\{\alpha_i^1\}$ denote the generators. 
Note that $\pi_1$ is the restriction to $M_1$ of the projection $M_0\times\R^{(k_0+1)^2}\to M_0$. 
According to definition \ref{def1}, we assume that $C_0\subset\Sing f$, which is described as follows. 
Let $v_1,...,v_n$ be Nash vector fields on $M_0$ which span the tangent space of $M_0$ at each point of 
$M_0$. 
Then we see, as previously, that $C_0\subset\Sing f$ if and only if there exist $C^\omega$ functions $a_{i,j}^0$ on 
$M_0$, $i=1,...,n,\ j=0,...,k_0$, such that 
\begin{equation}\tag{2}
v_if=\sum_{j=0}^{k_0}a_{i.j}^0h_j^0\quad\text{on}\ M_0\ \text{for}\ i=1,...,n.
\end{equation}

Let $\tilde h_0^0,...,\tilde h_{k_0}^0,\ \tilde g_1^0,...,\tilde g_{k'_0}^0,\,\tilde a_{i,j}^0,\,i=1,...,n,
\,j=0,...,k_0$, be $C^\omega$ approximations of $h_0^0,...,h_{k_0}^0$, $ g_1^0,...,g_{k'_0}^0,\,a_{i,j}^0,\,
i=1,...,n,\,j=0,...,k_0$, respectively, in the strong Whitney $C^\infty$ topology such that 
\begin{enumerate}
\item[($\tilde 1$)] $\sum_
{j=0}^{k_0}\tilde g_{i,j}^0\tilde h_j^0=0$ for $i=1,...,k'_0$ and 
\item[($\tilde 2$)] $ v_if=\sum_{j=0}^{k_0}\tilde a_
{i,j}^0\tilde h_j^0$ for $i=1,...,n$. 
\end{enumerate}
Then by lemma \ref{pert}, the common zero set $\tilde C_0=Z(\tilde h_j^0)$ is smooth and of the same dimension as 
$C_0$, $\tilde g_1^0,...,\tilde g_{k'_0}^0$ are generators of the sheaf of relations $\tilde{\mathcal A^0}$ of 
$\tilde h_0^0,...,\tilde h_{k_0}^0$, and if we let $M_0\times\R^{(k_0+1)^2}\supset M_0\times\PP(k_0)\supset
\tilde M_1\xrightarrow{\tilde\pi_1} M_0$ denote the blowing-up of $M_0$ along center $\tilde C
_0$ defined by $\tilde h_0^0,...,\tilde h_{k_0}^0$ then there exist analytic diffeomorphisms $\psi_0$ of 
$M_0$ and $\tilde\psi_0:M_1\to\tilde M_1$ close to id in the strong Whitney $C^\infty$ topology such that 
$\psi_0(C_0)=\tilde C_0$ and $\tilde\pi_1\circ\tilde\psi_0=\psi_0\circ\pi_1$. Hence $f\circ\tilde\pi_1:
\tilde M_1\to\R$ is close to $f\circ\pi_1:M_1\to\R$ in the strong Whitney $C^\infty$ topology because if we 
regard $f$ as a function on $M_0\times\R^{(k_0+1)^2}$ then $f\circ\tilde\pi_1=f|_{\tilde M_1}$ and $f\circ
\pi_1=f|_{M_1}$. 
Moreover, $\tilde C_0\subset\Sing f$ by $(\tilde2)$ for $\tilde h_j^0$, and $I(\tilde M_1)$ is generated by 
$\tilde l_{i_1,i_2,i_3}^0(x,y^0)$ which is defined by $\tilde l_{i_1,i_2,i_3}^0(x,y^0)=y_{i_1,i_2}^0\tilde h_{i_3}^0(x)-y_{i_3.i_1}^0
\tilde h_{i_2}^0(x),\ \sum_{j=0}^{k_0}y_{j,i_1}^0\tilde g_{i_2,j}^0(x)$ and $\xi_i^0(y^0)$ in $C^\omega(M_0
\times\R^{(k_0+1)^2})$. 
Let $\tilde\alpha_i^1$ denote the generators corresponding to $\alpha_i^1$.

Consider the second blowing-up $\pi_2:M_2\to M_1$ along $C_1$. 
In the same way as for the first blowing-up we define 
\begin{itemize}
\item $\mathcal I^1\subset\mathcal O_{M_1}$,
\item $h_0^1,...,h_{k_1}^1
\in H^0(M_1,\mathcal I^1)$,
\item $\mathcal A^1\subset\mathcal O_{M_1}^{k_1+1}$,
\item $ g_1^1=(g_{1,0}^1,...,g_{1,k_1}^1),...,g_{k'_1}^1=(g_{k'_1,0}^1,...,g_{k'_1,k_1}^1)\in H^0(M_1,\mathcal A^1)$,
\item $\xi_1^1,...,\xi_{s_1}^1\in\R[y_{i,j}^1]_{0\le i,j\le k_1}$, 
\item $l_{i_1,i_2,i_3}^1(x,y^0,y^1)=y_{i_1,i_2}^1h_{i
_3}^1(x,y^0)-y_{i_3.i_1}^1h_{i_2}^1(x,y^0)$ for $(x,y^0,y^1)\in M_1\times\R^{(k_1+1)^2}$ and for $i_1,i_2,
i_3=0,...,k_1$,
\item $C^\omega$ functions $a_{i,j}^1$ on $M_1$ for $i=1,...,n,\ j=0,...,k_1$ 
\end{itemize}
so that 

\begin{equation}\tag{1}
\sum_{j=0}^{k_1}g_{i,j}^1(x,y^0)h_j^1(x,y^0)=0\quad\text{on}\ M_1\ \text{for}\ i=1,...,k'_1,
\end{equation}
\begin{equation}\tag{2}
v_if(x)=\sum_{j=0}^{k_1}a_{i.j}^1(x,y^0)h_j^1(x,y^0)\quad\text{on}\ M_1\ \text{for}\ i=1,...,n,
\end{equation}
$$M_2=\{(x,y^0,y^1)\in M_1\times\R^{(k_1+1)^2}:l_{i_1,i_2,i_3}^1(x,y^0,y^1)=0,\ i_1,i_2,i_3=0,\ldots,k_1,$$
$$\sum_{j=0}^{k_1}y_{j,i_1}^1g_{i_2,j}^1(x,y^0)=0,\,i_1=0,...,k_1,\,i_2=1,...,k'_1,\,\text{and}\ \xi_i^1(y^1
)=0,\,i=1,...,s_1\},$$
where $I(M_2)$ in $C^\omega(M_1\times\R^{(k_1+1)^2})$ is generated by those functions, denoted by $\{\alpha_i^2\}$, 
in the last braces, and $\pi_2$ is the restriction to $M_2$ of the projection $M_1\times\R^{(k_1+1)^2}\to M_
1$.

Here we require as another prescription of blowings-up that $C_1$ is normal crossing with $\pi^{-1}_1(C_0)$. 
For each $(x,y^0)\in C_1\cap\pi_1^{-1}(C_0)$ there are two possible cases to consider\,: $C_1$ is transversal to 
$\pi^{-1}_1(C_0)$ at $(x,y^0)$ or $C_{1(x,y^0)}\subset\pi_1^{-1}(C_0)_{(x,y^0)}$. 
Divide, if necessary, $C_1$ into two unions $C_1^1$ and $C_1^2$ of its connected components so that on each union, only one case 
arises, and regard $\pi_1:M_1\to M_0$ as a composition $\pi_1=\pi_1^2
\circ \pi_1^1$, where $\pi_1^1$ is the blowing-up along center $C_1^1$
and $\pi_1^2$ is the blowing-up along center ${\pi_1^1}^{-1}(C_1^2)$.
Then we can assume that globally $C_1$ is transversal to $\pi_1^{-1}(C_0)$ or $C_1\subset\pi_1^{-1}(C_0)$. 
The latter case occurs if and only if there exist $C^\omega$ functions $b_{j_0,j_1}^1$ on $M_1$, $j_0=0,...,
k_0,\ j_1=0,...,k_1$, such that

\begin{equation}\tag{3}
h_{j_0}^0(x)=\sum_{j_1=0}^{k_1}b_{j_0,j_1}^1(x,y^0)h_{j_1}^1(x,y^0)\quad\text{on}\ M_1\ \text{for}\ j_0=0,...,
k_0.
\end{equation}
We extend $h_j^1,g_{i,j}^1$, $a_{i,j}^1$ and $b_{j_0,j_1}^1$ (if they
exist) to analytic functions on $M_0\times
\R^{(k_0+1)^2}$. 
For simplicity we use the same notations for the extensions. 
Then (1), (2) and (3) become
\begin{equation}\tag{1}
\sum_jg_{i,j}^1h_j^1=\sum_{i_1}\beta_{1,i_1,i}^1\alpha_{i_1}^1\quad\text{on}\ M_0\times\R^{(k_0+1)^2},
\end{equation}
\begin{equation}\tag{2}
v_if=\sum_ja_{i.j}^1h_j^1+\sum_{i_2}\beta_{2,i_2,i}^1\alpha_{i_2}^1\quad\text{on}\ M_0\times\R^{(k_0+1)^2},
\end{equation}
\begin{equation}\tag{3}
h_{j_0}^0=\sum_{j_1}b_{j_0,j_1}^1h_{j_1}^1+\sum_{i}\beta_{3,i,j_0}^1\alpha_{i}^1\quad\text{on}\ M_0\times\R^
{(k_0+1)^2}
\end{equation}
for some $C^\omega$ functions $\beta_{i_1,i_2,i_3}^i$ on $M_0\times\R^{(k_0+1)^2}$. 
Now $Z(h_j^1)\cap M_1=C_1$, and $I(M_2)$ in $C^\omega(M_0\times\R^{(k_0+1)^2}\times\R^{(k_1+1)^2})$ is 
generated by $\alpha_i^1$ and $\alpha_i^2$, which we naturally regard as $C^\omega$ functions on $M_0
\times\R^{(k_0+1)^2}\times\R^{(k_1+1)^2}$.

For the second blowing-up, we consider again $C^\omega$ approximations $\tilde h_0^1,...,\tilde h_{k_1}^1$, $\tilde g_1^1=(\tilde g_{1,0}^1,...,\tilde g^1_{1,k_1}),...,\tilde g_{k'_1}^1=(\tilde g_{k'_1,0}^1,...,
\tilde g_{k'_1,k_1}^1),\,\tilde a_{i,j}^1,\,i=1,...,n,\,j=0,...,k_1$, $\tilde b_{j_0,j_1}^1,\,j_0=0,...,k_0$,
$j_1=0,...,k_1$ and $\tilde\beta_{i_1,i_2,i_3}^1$ of $h_0^1,...,h_{k_1}^1,\,g_1^1=(g_{1,0}^1,...,g^1_{1,k_1}),
...,g_{k'_1}^1=(g_{k'_1,0}^1, ...,g_{k'_1,k_1}^1),\,a_{i,j}^1$, $b_{j_0,j_1}^1$ (if exist) and 
$\beta_{i_1,i_2,i_3}^1$ on $M_0\times\R^{(k_0+1)^2}$ in the strong Whitney $C^\infty$ topology such that
\begin{enumerate}
\item[($\tilde 1$)] $\sum_j\tilde g_{i,j}^1\tilde h_j^1=\sum_{i_1}\tilde\beta_{1,i_1,i}^1\tilde\alpha_{i_1}^1$ on $M_0
\times\R^{(k_0+1)^2}$ for $i=1,...,k'_1$, 
\item[($\tilde 2$)] $v_if=\sum_j\tilde a_{i,j}^1\tilde h_j^1+\sum_{i_2}
\tilde\beta_{2,i_2,i}^1\tilde\alpha_{i_2}^1$ on $M_0\times\R^{(k_0+1)^2}$ for $i=1,...,n$, 
\item[($\tilde 3$)] 
$\tilde h_{j_0}^0=\sum_{j_1}\tilde b_{j_0,j_1}^1\tilde h_{j_1}^1+\sum_{i}\tilde\beta_{3,i,j_0}^1\tilde\alpha_
{i}^1$ on $M_0\times\R^{(k_0+1)^2}$ for $j_0=0,...,k_0$. 
\end{enumerate}
Then $\tilde C_1=Z(\tilde h_j^1)\cap\tilde M_1$ is smooth and of the same dimension as $C_
1$. 
If $C_1\subset\pi_1^{-1}(C_0)$, then $\tilde C_1$ is contained in
$\tilde \pi_1^{-1}(\tilde C_0)$ by (3) and $(\tilde3)$. 
If $C_1\not\subset\pi_1^{-1}(C_0)$, i.e. if $C_1$ is transversal to $\pi_1^{-1}(C_0)$ in $M_1$, $\tilde C_1$ 
is transversal to $\tilde\pi_1^{-1}(\tilde C_0)$ in $\tilde M_1$ because the above diffeomorphism $\tilde\psi_0
:M_1\to\tilde M_1$ is close to id in the strong Whitney $C^\infty$ topology and carries $\pi_1^{-1}(C_0)$ to 
$\tilde\pi_1^{-1}(\tilde C_0)$ and because $\tilde C_1$ is close to $C_1$ in the strong Whitney $C^\infty$ 
topology. 
Hence, in any case $\tilde C_1$ is normal crossing with $\tilde\pi_1^{-1}(\tilde C_0)$. 
It also follows from $(\tilde2)$ that $\tilde C_1\subset\Sing f=\tilde\pi_1^{-1}(\Sing f)$. 
Thus we can take the blowing-up $\tilde\pi_2:\tilde M_2\to\tilde M_1$ of $\tilde M_1$ along center $\tilde C_1$, 
and embed $\tilde M_2$ by $\tilde h_0^1,...,\tilde h_{k_1}^1$ into $\tilde M_1\times\PP(k_1)\subset\tilde M_1
\times\R^{(k_1+1)^2}\subset M_0\times\R^{(k_0+1)^2}\times\R^{(k_1+1)^2}$ so that $\tilde\pi_2$ is the restriction 
to $\tilde M_2$ of the projection $\tilde M_1\times\R^{(k_1+1)^2}\to\tilde M_1$. 
Then there exist analytic diffeomorphisms $\psi_1:M_1\to\tilde M_1$ and $\tilde\psi_1:M_2\to\tilde M_2$ close 
to id in the strong Whitney $C^\infty$ topology ($\psi_1$ is not necessarily equal to $\tilde\psi_0$) such that 
$\psi_1(C_1)=\tilde C_1$ and $\tilde\pi_2\circ\tilde\psi_1=\psi_1\circ\pi_2$; $f\circ\tilde\pi_1\circ\tilde\pi_
2:\tilde M_2\to\R$ is close to $f\circ\pi_1\circ\pi_2:M_2\to\R$ in the strong Whitney $C^\infty$ topology; $I(
\tilde M_2)$ is generated in $C^\omega(M_0\times\R^{(k_0+1)^2}\times\R^{(k_1+1)^2})$ by $\tilde l_{i_1,i_2,i_3}
(x,y^0,y^1)=y_{i_1,i_2}^1\tilde h_{i_3}^1(x,y^0)-y_{i_3,i_1}^1\tilde h_{i_2}^1(x,y^0)$, 
$\sum_{j=0}^{k_1}y_{j,i_1}^1\tilde g_{i_2,j}^1(x,y^0)$, $\xi_i^1(y^1)$ and $\tilde\alpha_i^1$. 
Let $\tilde\alpha_i^2$ denote the former generators, and let $\tilde\alpha_i^1$ be naturally extended to $M_0
\times\R^{(k_0+1)^2}\times\R^{(k_1+1)^2}$.

Note that there exists a $C^\omega$ diffeomorphism from $M_2$ to $\tilde M_2$ close to id in the strong Whitney 
$C^\infty$ topology and carrying $\pi_2^{-1}(C_1)\cup(\pi_1\circ\pi_2)^{-1}(C_0)$ to $\tilde\pi_2^{-1}(\tilde C
_1)\cup(\tilde\pi_1\circ\tilde\pi_2)^{-1}(\tilde C_0)$ for the following reason. 
First by lemma 4.3 in \cite{FS}, we have a $C^\omega$ diffeomorphism from $M_1$ to $\tilde M_1$ close to id 
in the strong Whitney $C^\infty$ topology and carrying $C_1\cup\pi_1^{-1}(C_0)$ to $\tilde C_1\cup\tilde\pi_1^
{-1}(\tilde C_0)$. 
Hence we can assume that $C_1\cup\pi_1^{-1}(C_0)=\tilde C_1\cup\tilde\pi_1^{-1}(\tilde C_0)$. 
Then in the same way as in the proof of lemma \ref{pert} we construct a $C^\omega$ diffeomorphism $\eta:M_2\to\tilde M_2$ 
close to id in the strong Whitney $C^\infty$ topology such that $\tilde\pi_2\circ\eta=\pi_2$ and hence $\eta(
\pi_2^{-1}(C_1)\cup(\pi_1\circ\pi_2)^{-1}(C_0))=\eta(\pi_2^{-1}(C_1\cup\pi_1^{-1}(C_0)))=\tilde\pi_2^{-1}(C_1
\cup\pi_1^{-1}(C_0))=\tilde\pi_2^{-1}(\tilde C_1)\cup(\tilde\pi_1\circ\tilde\pi_2)^{-1}(\tilde C_0)$.

We repeat the same arguments inductively on each blowing-up. 
Then condition (3) becomes somewhat complicated because the union of the inverse images of the previous 
centers is not necessarily smooth. 
Let us consider the center $C_2$ of the blowing-up $\pi_3:M_3\to M_2$. 
We describe the condition that $C_2$ is normal crossing with $A=\pi_2^{-1}(C_1)\cup(\pi_
1\circ\pi_2)^{-1}(C_0)$ as follows. 
Let $h_j^2,\,g_j^2=(g_{j,0}^2,...,g_{j,k_2}^2),\,\xi_i^2,\,a_{i,j}^2,\,\beta_{i_1,i_2,i_3}^2$ on $M_0\times\R^
{(k_0+1)^2}\times\R^{(k_1+1)^2}$ and their $C^\omega$ approximations $\tilde h_j^2,\,\tilde g_j^2=(\tilde g_{j
,0}^2,...,\tilde g_{j,k_2}^2),\,\tilde\xi_i^2,\,\tilde a_{i,j}^2,\,\tilde\beta_{i_1,i_2,i_3}^2$ in the strong 
Whitney $C^\infty$ topology be given for the center $C_2$ so that the corresponding equalities to (1), 
($\tilde 1$), (2) and ($\tilde2$) hold. 
Set $Y=\overline{A-\pi_2^{-1}(C_1)}$. 
Then $Y$ is a smooth analytic set of codimension 1 in $M_2$. If $C_1\not\subset\pi_1^{-1}(C_0)$ then $Y=A$; 
if $C_1\subset\pi_1^{-1}(C_0)$ then $\pi_2^{-1}(C_1)\subset(\pi_1\circ\pi_2)^{-1}(C_0)$ and $Y\cup\pi_2^{-1}
(C_1)$ is a decomposition of $(\pi_1\circ\pi_2)^{-1}(C_0)$ to smooth analytic sets, it follows from the 
normal crossing property of $C_2$ with $A$ that $C_2$ is normal crossing with $\pi_2^{-1}(C_1)$ and with $Y$ 
(the converse is not necessarily correct), and $I(Y)$ in
$C^\omega(M_2)$ coincides with 
$$\{g\in C^\omega(M_2):
gh_i^1=\sum_{j=0}^{k_0}c_{i,j}h_j^0 \textrm{ for some } c_{i,j}\in C^\omega(M_2),\ i=0,...,k_1\}.$$ 
Let $\chi_j^{0,2},\ j=1,...,t_2$, be generators of $I(Y)$. 
Then there exist $C^\omega$ functions $c_{j_0,j_1,j}^{0,2}$ on $M_2,\ j_0=0,...,k_0,\ j_1=1,...,k_1,\ j=1,...,
t_2$, such that 
\begin{equation}\tag{4}
\chi_j^{0,2}(x,y^0,y^1)h_{j_1}^1(x,y^0)=\sum_{j_0=0}^{k_0}c_{j_0,j_1,j}^{0,2}(x,y^0,y^1)h_{j_0}^0(x)\quad\text
{on}\ M_2,
\end{equation}
and as in the case of the second blowing-up, dividing $C_2$ if
necessary we can assume that $C_2$ is transversal to $\pi_2
^{-1}(C_1)$ or contained in $\pi_2^{-1}(C_1)$ and that $C_2$ is transversal to $Y$ or contained in $Y$. 
If $C_2\subset\pi_2^{-1}(C_1)$ then there exist $C^\omega$ functions $b_{j_1,j_2}^2$ on $M_2$, $j_1=0,...,k_1,
\ j_2=0,...,k_2$, such that 
\begin{equation}\tag{3}
h_{j_1}^1(x,y^0)=\sum_{j_2=0}^{k_2}b_{j_1,j_2}^2(x,y^0,y^1)h_{j_2}^2(x,y^0,y^1)\quad\text{on}\ M_2,
\end{equation}
and if $C_2\subset Y$ then there exist $C^\omega$ functions $d_{j,j_2}^{0,2}$ on $M_2,\ j=1,...,t_2,\ j_2=0,...
,k_2$, such that 
\begin{equation}\tag{5}
\chi_j^{0,2}(x,y^0,y^1)=\sum_{j_2=0}^{k_2}d_{j,j_2}^{0,2}(x,y^0,y^1)h_{j_2}^2(x,y^0,y^1)\quad\text{on}\ M_2. 
\end{equation}
As before we assume that $\chi_j^{0,2},\,b_{j_1,j_2}^2,\,c_{j_1,j_2}^{0,2},\,d_{j,j_2}^{0,2}$ are defined on $M_0
\times\R^{(k_0+1)^2}\times\R^{(k_1+1)^2}$. 
Then there exist $C^\omega$ functions $\gamma_{i_1,i_2,i_3}^i$ and $\gamma_{i_1,i_2,i_3,i_4}^i$ on $M_0\times
\R^{(k_0+1)^2}\times\R^{(k_1+1)^2}$ such that 
\begin{equation}\tag{3}
h_{j_1}^1=\sum_{j_2}b_{j_1,j_2}^2h_{j_2}^2+\sum_i\gamma_{1,i,j_1}^1\alpha_i^1+\sum_i\gamma_{1,i,j_1}^2\alpha_i
^2, 
\end{equation}
\begin{equation}\tag{4}\chi_j^{0,2}h_{j_1}^1=\sum_{j_0}c_{j_0,j_1,j}^{0,2}h_{j_0}^0+\sum_i\gamma_{2,i,j,j_1}^1\alpha_i^1+\sum_i\gamma
_{2,i,j,j_1}^2\alpha_i^2,
\end{equation} 
\begin{equation}\tag{5}\chi_j^{0,2}=\sum_{j_2}d_{j,j_2}^{0,2}h_{j_2}^2+\sum_i\gamma_{3,i,j}^1\alpha_i^1+\sum_i\gamma_{3,i,j}^2\alpha_
i^2\quad\text{on}\ M_0\times\R^{(k_0+1)^2}\times\R^{(k_1+1)^2}.
\end{equation} 

We need to consider also $C^\omega$ approximations $\tilde\chi_j^{0,2},\,\tilde b_{j_1,j_2}^2,\,\tilde c_{j_
1,j_2,j}^{0,2},\,\tilde d_{j,j_2}^{0,2},\,\tilde\gamma_{i_1,i_2,i_3}^i$ and $\tilde\gamma_{i_1,i_2,i_3,i_4}^i$ 
of $\chi_j^{0,2},\,b_{j_1,j_2}^2$ (if they exist),$\,c_{j_1,j_2,j}^{0,2},\,d_{j,j_2}^{0,2}$ (if exist), $\gamma_{i_
1,i_2,i_3}^i$ (if exist) and $\gamma_{i_1,i_2,i_3,i_4}^i$ on $M_0\times\R^{(k_0+1)^2}\times\R^{(k_1+1)^2}$ in 
the strong Whitney $C^\infty$ topology such that 
\begin{enumerate}
\item[($\tilde 3$)] $\tilde h_{j_1}^1=\sum_{j_2=0}^{k_2}\tilde b_{j_1,
j_2}^2\tilde h_{j_2}^2+\sum_i\tilde\gamma_{1,i,j_1}^1\tilde\alpha_i^1+\sum_i\tilde\gamma_{1,i,j_1}^2\tilde\alpha
_i^2,$ 
\item[($\tilde 4$)] $\tilde\chi_j^{0,2}\tilde h_{j_1}^1=\sum_{j_0}\tilde c_{j_0,j_1,j}^{0,2}\tilde h_{j_0}^0+\sum
_i\tilde\gamma_{2,i,j,j_1}^1\tilde\alpha_i^1+\sum_i\tilde\gamma_{2,i,j,j_1}^2\tilde\alpha_i^2$,
\item[($\tilde 5$)]$ \tilde\chi_j^{0,2}=\sum_{j_2}\tilde d_{j,j_2}^{0,2}\tilde h_{j_2}^2+\sum_i\tilde\gamma_{3,i,j}^1
\tilde\alpha_i^1+\sum_i\tilde\gamma_{3,i,j}^2\tilde\alpha_i^2$ on $M_0\times\R^{(k_0+1)^2}\times\R^{(k_1+1)^2}$. 
\end{enumerate}
Set $\tilde Y=Z(\tilde\chi_j)\cap\tilde M_2$. 
Then 
$$\tilde Y=\overline{(\tilde\pi_1\circ\tilde\pi_2)^{-1}(\tilde C_0)-\tilde\pi_2^{-1}}
\overline{(\tilde C_1)}$$ 
because $\tilde Y$ contains the right hand side by ($\tilde4$) and because the opposite inclusion follows from the facts 
that $Y$ and the right hand side are smooth and of codimension 1 in $M_2$ and in $\tilde M_2$, respectively, and that $\chi_j$ 
are generators of $I(Y)$ in $C^\omega(M_2)$. 
Hence $\tilde\pi_2^{-1}(\tilde C_1)\cup(\tilde\pi_1\circ\pi_2)^{-1}(\tilde C_0)$, which is normal crossing, is 
the union of the smooth analytic sets $\tilde\pi_2^{-1}(\tilde C_1)$ and $\tilde Y$. 
Moreover, $\tilde C_2$ is normal crossing with $\tilde\pi_2^{-1}(\tilde C_1)\cup(\tilde\pi_1\circ\pi_2)^{-1}(
\tilde C_0)$ for the following four reasons. 

If $C_2$ is transversal to $\pi_2^{-1}(C_1)$ or to $Y$, so is $\tilde C_2$ to $\tilde\pi_2^{-1}(\tilde C_1)$ 
or to $\tilde Y$, respectively, by the same reason as before. 
If $C_2\subset\pi_2^{-1}(C_1)$, then there exist $C^\omega$ functions $b_{j_1,j_2}^2$ with (3) on $M_2$, hence $\tilde h
_{j_1}^1=\sum_{j_2=0}^{k_2}\tilde b_{j_1,j_2}^2\tilde h_{j_2}^2$ on $\tilde M_2$ and $\tilde C_2\subset\tilde\pi
_2^{-1}(\tilde C_1)$. 
In the same way we see that if $C_2\subset Y$ then $\tilde C_2\subset\tilde Y$. 
The fourth consideration is that $C_2$ is normal crossing with $\pi_2^{-1}(C_1)\cup(\pi_1\circ\pi_2)^{-1}(C_0)$. 

By these four properties we can find also a $C^\omega$ diffeomorphism from $M_2$ to $\tilde M_2$ close to id in 
the strong Whitney $C^\infty$ topology and carrying $C_2,\ \pi_2^{-1}(C_1)$ and $(\pi_1\circ\pi_2)^{-1}(C_0)$ 
to $\tilde C_2,\ \tilde\pi_2^{-1}(\tilde C_1)$ and $(\tilde\pi_1\circ\tilde\pi_2)^{-1}(\tilde C_0)$, 
respectively.

Let $1<m'\,(<m)\in\N$. 
As above we inductively embed $M_{m'}$ into $M_{m'-1}\times\R^{(k_{m'-1}+1)^2}(\subset M_0\times\R^{(k_0+1)^2}
\times\cdots\times\R^{(k_{m'-1}+1)^2})$ and obtain a finite number of $C^\omega$ functions on $M_0\times\cdots
\times\R^{(k_{m'-1}+1)^2}$, namely $h_j^{m'},\,a_{i,j}^{m'},\,\xi_i^{m'},\,\chi_j^{m'',m'},\,c_{j_0,j_1,j}^{m'',m'},\,
d_{j,j_1}^{m'',m'},\,\alpha_i^{m'},\,\beta_{i_1,i_2,i_3}^{m'',m'},\,\beta_{i_1,
i_2,i_3,i_4}^{m''',m'',m'}$, $\beta_{i_1,i_2,i_3}^{m''',m'',m'}$ for $m''\,(<m'),m'''\,(\le m')
\in\N$ and a finite number of $C^\omega$ maps from $M_0\times\cdots\times\R^{(k_{m'-1}+1)^2}$ to $\R^{(k_{m'}+
1)^2}$, namely $g_j^{m'}=(g_{j,0}^{m'},...,g_{j,k_{m'}}^{m'})$ such that the following conditions are satisfied:
\begin{itemize}
\item the blowing-up $\pi_{m'}:M_{m'}\to M_{m'-1}$ is the restriction to $M_{m'}$ of the projection $M_{m'-1}\times\R
^{(k_{m'-1}+1)^2}\to M_{m'-1}$;
\item  $\{h_j^{m'}:j\}$ are generators of $I(C_{m'})$ in $C^\omega(M_{m'})$; 
\item $\{\xi_i^
{m'}(y^{m'}):i\}$ are generators of $I(\PP(k_{m'}))$ in $\R[y_{i,j}^{m'}]_{0\le i,j\le k_{m'}}\,(\subset C^\omega
(\R^{(k_{m'}+1)^2}))$; 
\item $\{g_j^{m'}:j\}$ are generators of the sheaf of relations of $h_0^{m'},...,h_{k_{m'}}^{m'
}$ on $M_{m'}$;
\item  $\{\chi_j^{m'',m'}:j\}$ are generators of $I(Y_{m'',m'})$ in $C^\omega(M_{m'})$, where $Y_{m'-1,
m'}=\pi_{m'}^{-1}(C_{m'-1})$ and 
$$Y_{m'',m'}\!=\!\overline{(\pi_{m''+1}\circ\cdots\circ\pi_{m'})^{-1}(C_{m''})-
(\pi_{m''+2}\circ\cdots\circ\pi_{m'})^{-1}(C_{m''+1})-\cdots-\pi_{m'}^{-1}(C_{m'-1})}
$$
for $m''<m'-1$; $\{\alpha_i^{m'}\!:\!i\}\!=\!\{y_{i_1,i_2}^{m'-1}h_{i_3}^{m'-1}-y_{i_3,i_1}^{
m'-1}h_{i_2}^{m'-1},\,\sum_jy_{j,i_1}^{m'-1}g_{i_2,j}^{m'-1},\,\xi_i^{m'-1}:i_1,i_2,i_3,i\}$;
\item  $\{\alpha_i^{m'''
}:m'''\le m',\,i\}$ are generators of $I(M_{m'})$ in $C^\omega(M_0\times\cdots\times\R^{(k_{m'-1}
+1)^2})$, where we naturally regard $h_i^{m'-1},\,g_{i,j}^{m'-1},\,\xi_i^{m'-1}$ and $\alpha_i^{m'''}$ as 
functions on $M_0\times\cdots\times\R^{(k_{m'-1})^2}$;
\item  (1) $\sum_jg_{i,j}^{m'}h_j^{m'}\!=\sum_{_{\substack{m'''\le m'\\ 
{i_1}}}}\beta_{1,i_1,i}^{m''',m'}\alpha_{i_1}^{m'''}$ on $M_0\times\cdots\times\R^{(k_{m'-1}+1)^2}$;
\item  (2) $v_if=
\sum_ja_{i,j}^{m'}h_j^{m'}+\sum_{_{\substack{m'''\le m'\\ i_2}}}\beta_{2,i_2,i}^{m''',m'}\alpha_{i_2}^{
m'''}$ on $M_0\times\cdots\times\R^{(k_{m'-1}+1)^2}$;
\item  (4) $\chi_j^{m'',m'}h_{j_1}^{m'-1}=\sum_{j_0}c_{j_0,j_1,j
}^{m'',m'}\chi_{j_0}^{m'',m'-1}+\sum_{_{\substack{m'''\le m'\\ i}}}\beta_{4,i,j,j_1}^{m''',m'',m'}\alpha_i^{m'''}$ on 
$M_0\times\cdots\times\R^{(k_{m'-1}+1)^2}$ for $m''<m'-1$; 
\item (5) $\chi_j^{m'',m'}\!=\!\sum_{j_1}d_{j,j_1}^{m'',m'}h_{j_1}^{m'}+\sum_{_{\substack{m'''\le m'
\\ i}}}\beta_{5,i,j}^{m''',m'',m'}\alpha_i^{m'''}$ on $M_0\times\cdots\times\R^{(k_{m'-1}+1)^2
}$ for $m''<m'-1$ if $C_{m'}\subset Y_{m'',m'}$; 
\item (6) $\chi_j^{m'-1,m'}=h_j^{m'-1}$ on $M_0\times\cdots
\times\R^{(k_{m'-1}+1)^2}$. 
\end{itemize}
(Condition (3) is included in (5) and (6).) 
Here $d_{j,j_1}^{m'',m'}$ and $\beta_{5,i,j}^{m''',m'',m'}$ exist only if $C_{m'}\subset Y_{m'',m'}$ and we 
assume that if $C_{m'}\not\subset Y_{m'',m'}$ then $C_{m'}$ is transversal to $Y_{m'',m'}$ in $M_{m'}$. 
Note that $\cup_{m''<m'}Y_{m'',m'}$ is a decomposition of $\pi_{m'}^{-1}(C_{m'-1})\cup\cdots
\cup(\pi_1\circ\cdots\circ\pi_{m'})^{-1}(C_0)$ into smooth analytic sets.

Assume, inductively, that there exist a blowing-up $M_0\times\cdots\times\R^{(k_{m'-1}+1)^2}\supset\tilde M
_{m'}\xrightarrow{\tilde\pi_{m'}} \tilde M_{m'-1}$ along center $\tilde C_{m'-1}$ 
close to $M_0\times\cdots\times\R^{(k_{m'-1}+1)^2}\supset M_{m'}\xrightarrow{\pi_{m'}} M_{m'-1}$ in 
the strong Whitney $C^\infty$ topology and $C^\omega$ approximations 
$$\tilde h_j^{m'},\,\tilde a_{i,j}^{m'},
\tilde\chi_j^{m'',m'},\,\tilde c_{j_0,j_1,j}^{m'',m'},\,\tilde d_{j,j_1}^{m'',m'},\,\tilde\alpha
_i^{m'}, \tilde\beta_{i_1,i_2,i_3}^{m'',m'}, \tilde\beta_{i_1,i_2,i_3,i_4}^{m''',m'',m'},\tilde\beta_{i_1,i_2,i_3}
^{m''',m'',m'},\tilde g_j^{m'}\!=\!(\tilde g_{j,0}^{m'},...,\!\tilde g_{j,k_{m'}}^{m'})$$ 
of 
$$h_j^{m'},\,a_{i,
j}^{m'},\,\chi_j^{m'',m'},\,c_{j_0,j_1,j}^{m'',m'},\,d_{j,j_1}^{m'',m'},\,\alpha_i^{m'}, \beta_{i_1,i_2,i_3}^
{m'',m'}, \beta_{i_1,i_2,i_3,i_4}^{m''',m'',m'},\beta_{i_1,i_2,i_3}^{m''',m'',m'},g_j^{m'}\!=\!(g_{j,0}^
{m'},...,g_{j,k_{m'}}^{m'})$$
on $M_0\times\cdots\times\R^{(k_{m'-1}+1)^2}$ in the strong Whitney $C^\infty$ 
topology such that $\tilde\pi_{m'}$ is the restriction to $\tilde M_{m'}$ of the projection $M_0\times\cdots
\times\R^{(k_{m'-1}+1)^2}\to M_0\times\cdots\times\R^{(k_{m'-2}+1)^2}$, 
$$\{\tilde\alpha_i^{m'}:i\}=\{y_{i_1,i_2}
^{m'-1}\tilde h_{i_3}^{m'-1}-y_{i_3,i_1}^{m'-1}\tilde h_{i_2}^{m'-1},\,\sum_jy_{j,i_1}^{m'-1}
\tilde g_{i_2,j}^{m'-1},\,\xi_i^{m'-1}:i_1,i_2,i_3,i\}$$ and the corresponding conditions $(\tilde1),\,(\tilde2),
\,(\tilde4),\,(\tilde5)$ (if $C_{m'}\subset Y_{m'',m'}$) and $(\tilde6$) to (1), (2), (4), (5) (if $C_{m'}
\subset Y_{m'',m'}$) and (6) are satisfied. 
Set $\tilde C_{m'}=Z(\tilde h_j^{m'})\cap\tilde M_{m'}$, $\tilde Y_{m'-1,m'}=\tilde\pi_{m'}^{-1}(\tilde C_{m'-1
})$ and 
$$\tilde Y_{m'',m'}=\overline{(\tilde\pi_{m''-1}\circ\cdots\circ\tilde\pi_{m'})^{-1}(\tilde C_{m''})-
\tilde(\pi_{m''}\circ\cdots\circ\tilde\pi_{m'})^{-1}(\tilde C_{m''+1})-\cdots-
\tilde\pi_{m'}^{-1}(\tilde C_{m'-1})}$$ for $m''<m'-1$. 
Then, as above, we have:
\begin{itemize}
\item $I(\tilde M_{m'})$ in $C^\omega(M_0\times\cdots\times\R^{(k_{m'-1}+1)^2})$ is generated by 
$\{\tilde\alpha_i^{m''}:m''\le m',\,i\}$; 
\item $\tilde C_{m'}$ is smooth and of the same dimension as $C_{m'}$; 
\item  $I(\tilde C_{m'})$ in $C^\omega(\tilde M_{m'})$ is generated by $\{\tilde h_j^{m'}:j\}$; 
\item  $\{\tilde g_j^{m'}:j
\}$ are generators of the sheaf of relations of $\tilde h_0^{m'},...,\tilde h_{k_{m'}}^{m'}$ on $\tilde M_{m'}
$; 
\item  $I(\tilde Y_{m'',m'})$ in $C^\omega(\tilde M_{m'})$ for each $m''<m'$ is generated by $\{\tilde\chi_j^{m'',
m'}:j\}$ by ($\tilde4$); 
\item  $\tilde C_{m'}\subset\tilde Y_{m'',m'}$ if and only if $C_{m'}\subset Y_{m'',m'}$;  
\item if 
$C_{m'}\not\subset Y_{m'',m'}$ then $\tilde C_{m'}$ is transversal to $\tilde Y_{m'',m'}$ in $\tilde M_{m'}$; 
\item  
$\cup_{m''<m'}\tilde Y_{m'',m'}$ is a decomposition of $\tilde\pi_{m'}^{-1}(\tilde C_{m'-1})\cup\cdots\cup(
\tilde\pi_1\circ\cdots\circ\tilde\pi_{m'})^{-1}(\tilde C_0)$ into smooth analytic sets;  
\item $\tilde C_{m'}$ is 
normal crossing with this set;  
\item there exists a $C^\omega$ diffeomorphism from $M_{m'}$ to $\tilde M_{m'}$ close 
to id in the strong Whitney $C^\infty$ topology and carrying $C_{m'},...,(\pi_1\circ\cdots\circ\pi_{m'})^{-1}(
C_0)$ to $\tilde C_{m'},...,(\tilde\pi_1\circ\cdots\circ\tilde\pi_{m'})^{-1}(\tilde C_0)$, respectively;  
\item $f
\circ\tilde\pi_1\circ\cdots\circ\tilde\pi_{m'}$ is close to $f\circ\pi_1\circ\cdots\circ\pi_{m'}$ in the strong 
Whitney $C^\infty$ topology. 
\end{itemize}

Finally, as above, we embed $M_m$ and $\tilde M_m$ into $M_0\times\R^{(k_0+1)^2}\times\cdots\times\R^{(k_{m-1}
+1)^2}$ by $h_0^{m-1},...,h_{k_{m-1}}^{m-1}$ and $\tilde h_0^{m-1},...,\tilde h_{k_{m-1}}^{m-1}$, respectively, 
define $\alpha_i^m,\tilde\alpha_i^m,Y_{m',m}$ and $\tilde Y_{m',m}$ for $0\le m'<m$, and let $\{\chi_j^{m',m}:j\}$ 
and $\{\tilde\chi_j^{m',m}:j\}$ be finitely many $C^\omega$ functions on $M_0\times\cdots\times\R^{(k_{m-1}+1)^2}$ 
which are generators of $I(Y_{m',m})$ in $C^\omega(M_m)$ and of $I(\tilde Y_{m',m})$ in $C^\omega(\tilde M_m)$, 
respectively, for each $m'<m$ such that each $\tilde\chi_j^{m',m}$ is close to $\chi_j^{m',m}$ in the strong 
Whitney $C^\infty$ topology. 
Then there exists a $C^\omega$ diffeomorphism $\psi_m:M_m\to\tilde M_m$ close to id in the strong Whitney $C^
\infty$ topology carrying $\pi_m^{-1}(C_{m-1}),...,(\pi_1\circ\cdots\circ\pi_m)^{-1}(C_0)$ to $\tilde\pi_m^{-1}(
\tilde C_{m-1}),...,(\tilde\pi_1\circ\cdots\circ\tilde\pi_m)^{-1}(\tilde C_0)$, respectively. 
Set $F=f\circ\pi_1\circ\cdots\circ\pi_m$ and $\tilde F=f\circ\tilde\pi_1\circ\cdots\circ\tilde\pi_m$. 
Then $F$ has only normal crossing singularities. 
We require $\psi_m$ to carry, moreover, $\Sing F$ to $\Sing\tilde F$. 
That is possible if $\tilde F$ has only normal crossing singularities by the same reason as before.

We will describe a condition for $\tilde F$ to have only normal crossing singularities. 
As the problem in the theorem is local around the compact subset $X$ of $M$, we assume that $M_m$ is covered by a 
finite number of good open subsets in the following sense. 
We have the disjoint union $B$ of finitely many closed balls $B_i$ in the Euclidean space of same dimension as $M$, 
a $C^\omega$ immersion $\rho=(\rho_{-1},...,\rho_{m-1}):B\to M_0\times\R^{(k_0+1)^2}\times\cdots\times\R^{(k_{m-1}
+1)^2}$, finitely many $C^\omega$ functions $\delta_{i,j}$ on each $B_i$ regular at $\delta_{i,j}^{-1}(0)$ and 
$\theta_{i,j}>0\in\N$ such that $\Ima \rho\subset M_m,\ \rho(\Int B)\supset X$, for each $i$ $\rho|_{B_i}$ is an 
embedding, $F\circ\rho|_{\Int B_i}$ has only normal crossing singularities with unique singular value $z_{0i}$, and 
$$
F\circ\rho|_{B_i}=\prod_j\delta_{i,j}^{\theta_{i,j}}+z_{0i}. 
$$
Here the condition $\Ima\rho\subset M_m$ and the last condition are equivalent to
\begin{equation}\tag{7}
f\circ\rho_{-1}|_{B_i}=\prod_j\delta_{i,j}^{\theta_{i,j}}+z_{0i}
\end{equation}
and there exist $C^\omega$ functions $\nu_{i,j}^{m'}$ on $M_0\times\cdots\times\R^{(k_{m-1}+1)^2}\times B$ such 
that for each $\alpha_i^{m'}$ with $m'\le m$ 
\begin{equation}\tag{8}
\alpha_i^{m'}(x,y^0,...,y^{m'-1})  = \nu_{i,-1}^{m'}(x,y^0,...,y^{m-1},z)(x-\rho_{-1}(z))+\end{equation}
$$\sum_{j=0}^{m-1}\nu_{i,j}^{m'}(x,y^0,...,y^{m-1},z)(y^j-\rho_j(z))\text{ on}\ M_0\times\cdots\times\R^{(k_{m-1}+1)^2}\times B $$
because $x-\rho_{-1}(z),\ y^j-\rho_j(z),\,j=0,...,m-1$, generate the ideal of $C^\omega(M_0\times\cdots\times\R^{
(k_{m-1}+1)^2}\times B)$ defined by the graph of $\rho$---$\{(\rho(z),z):z\in B\}$. 
Conversely, the existence of such $\rho,\,\delta_{i,j},\,\theta_{i,j}$ and $\nu_{i,j}^{m'}$ implies the normal crossing 
property of $F$. 
Note that
$$
\{z_{0i}\}=F(\Sing F|_{(\pi_1\circ\cdots\circ\pi_m)^{-1}(U)})=f(\Sing f|_U)
$$
for an open neighborhood $U$ of $X$ in $M$. 
(Assume that $U=M$ for simplicity of notation.) 
Hence a condition for $\tilde F$ to have only normal crossing singularities is that there exist $C^\omega$ 
approximations $\tilde\rho=(\tilde\rho_{-1},...,\tilde\rho_{m-1}):B\to M_0\times\R^{(k_0+1)^2}\times\cdots\times\R^
{(k_{m-1}+1)^2}$ of $\rho,\ \tilde\delta_{i,j}$ of $\delta_{i,j}$ and $\tilde\nu_{i,j}^{m'}$ of $\nu_{i,j}^{m'}$ in 
the strong Whitney $C^\infty$ topology such that 
\begin{equation}\tag{$\widetilde 7$}
f\circ \tilde \rho_{-1}|_{B_i}=\prod_j\tilde \delta_{i,j}^{\theta_{i,j}}+z_{0i},
\end{equation}
\begin{equation}\tag{$\widetilde 8$}
\tilde {\alpha_i}^{m'}(x,y^0,...,y^{m'-1})=\tilde{\nu}_{i,-1}^{m'}(x,y^0,...,y^{m-1},z)(x-\tilde\rho_{-1}(z))+\end{equation}
$$\sum_{j=0}^{m-1}\tilde\nu_{i,j}^{m'}(x,y^0,...,y^{m-1},z)(y^j-\tilde\rho_j(z))\quad\text{on}\ M_0\times
\cdots\times\R^{(k_{m-1}+1)^2}\times B.$$ 

However, we cannot find the approximations directly by proposition \ref{prop-app} below. Indeed, we need additional arguments as follows. 
Extend trivially $\rho$ to $\rho=(\rho_{-1},...,\rho_{m-1}):M_0\times\cdots\times\R^{(k_{m-1}+1)^2}\times B\to M_0
\times\cdots\times\R^{(k_{m-1}+1)^2}$ and $\delta_{i,j}$ to $\delta_{i,j}:M_0\times\cdots\times\R^{(k_{m-1}+1)^2}
\times B_i\to\R$. 
Then (7) holds on $M_0\times\cdots\times\R^{(k_{m-1}+1)^2}\times B_i$. 
Approximate these extended $\rho$ and $\delta_{i,j}$ by a $C^\omega$ map $\tilde\rho=(\tilde\rho_{-1},...,\tilde\rho
_{m-1}):M_0\times\cdots\times\R^{(k_{m-1}+1)^2}\times B\to M_0\times\cdots\times\R^{(k_{m-1}+1)^2}$ and $C^\omega$ 
functions $\tilde\delta_{i,j}$ on $M_0\times\cdots\times\R^{(k_{m-1}+1)^2}\times B_i$, respectively, so that 
$(\tilde7)$ and $(\tilde8)$ hold on $M_0\times\cdots\times\R^{(k_{m-1}+1)^2}\times B$ and $M_0\times\cdots\times
\R^{(k_{m-1}+1)^2}\times B_i$, respectively. 
Regard $M_0$ locally as a Euclidean space, and consider the map $\tilde P:M_0\times\cdots\times\R^{(k_{m-1}+1)^2}
\times B\ni(x,y^0,...,y^{m-1},z)\to(x-\tilde\rho_{-1}(x,...,z),...,y^{m-1}-\tilde\rho_{m-1}(x,...,z))\in M_0\times
\cdots\times\R^{(k_{m-1}+1)^2}$. 
As $\tilde P$ is close to the map $:M_0\times\cdots\times\R^{(k_{m-1}+1)^2}\times B\ni(x,y^0,...,y^{m-1},z)\to(x-
\rho_{-1}(x,...,z),...,y^{m-1}-\rho_{m-1}(x,...,z))\in M_0\times\cdots\times\R^{(k_{m-1}+1)^2}$, the Jacobian 
matrix $\vmatrix\frac{D(\tilde P)}{D(x,...,y^{m-1})}\endvmatrix$ vanishes nowhere. 
Hence by the implicit function theorem we have an analytic map $\hat\rho=(\hat\rho_{-1},...,\hat\rho_{m-1}):B\to 
M_0\times\cdots\times\R^{(k_{m-1}+1)^2}$ such that $\tilde\rho(\hat\rho(z),z))=\hat\rho(z)$ and $\hat\rho$ is close 
to $\rho$ in the strong Whitney topology. 
Then $\hat\rho$ is a $C^\omega$ immersion, 
\begin{equation}\tag{$\hat 7$}
f\circ\hat\rho_{-1}(z)=f\circ\tilde\rho_{-1}(\hat\rho(z),z)=\prod_j\tilde\delta_{i,j}^{\theta_{i,j}}(\hat\rho(z),z)+
z_{0i}\quad\text{for}\ z\in B_i,
\end{equation}
\begin{equation}\tag{$\hat 8$}\tilde\alpha_i^{m'}\circ\hat\rho(z)=\tilde\nu_{i,-1}^{m'}(\hat\rho(z),z)(\hat\rho_{-1}(z)-\tilde\rho_{-1}(\hat\rho_{-1}(z),z))+\end{equation}
$$\sum_{j=0}^{m-1}\tilde\nu_{i,j}^{m'}(\hat\rho(z),z)(\hat\rho_j(z)-\tilde\rho_j(\hat\rho(z),z))=0\quad
\text{for}\ z\in B.$$

By $(\hat8)$, $\Ima\hat\rho\subset\tilde M_m$, hence $\rho(\Int B)\supset X$, and by $(\hat7)$, $\tilde F$ has only 
normal crossing singularities because $\tilde\delta_{i,j}(\hat\rho(z),z)$ is close to $\delta_{i,j}(z)$ in the 
strong Whitney $C^\infty$ topology. 
Note that if $\tilde\rho$ is of class Nash, so is $\hat\rho$.

Under the conditions $(\tilde7)$ and $(\tilde8)$, $F$ and $\tilde F$ are $C^\omega$ right equivalent through a 
$C^\omega$ diffeomorphism close to id in the strong Whitney $C^\infty$ topology for the following reason. 
Since $F$ and $\tilde F$ have only normal crossing singularities, and since $f\circ\rho_{-1}$ and $f\circ\hat\rho_
{-1}$ are $C^\omega$ right equivalent by (7) and $(\hat7)$, we can modify $\psi_m$ to carry $\Sing F$ to $\Sing
\tilde F$ (cf. step 1 of the proof of theorem 3.1 in \cite{FS}).
Replacing $\tilde F$ with $\tilde F\circ\psi_m$, we assume that $\tilde M_m=M_m$, $\tilde\pi_m^{-1}(\tilde C_{m-1})
=\pi_m^{-1}(C_{m-1}),...,(\tilde\pi_1\circ\cdots\circ\tilde\pi_m)^{-1}(\tilde C_0)=(\pi_1\circ\cdots\circ\pi_m)^{
-1}(C_0)$ and $\Sing F=\Sing\tilde F$. 
Let $\kappa$ be a Nash function on $\R$ with zero set $\{z_{0i}\}$ and regular there. 
Then $\kappa\circ F$ and $\kappa\circ\tilde F$ satisfy the assumption
of lemma 4.7 in \cite{FS}:
\begin{itemize}
\item they have the same sign at each 
point of $M$, only normal crossing singularities at $(\kappa\circ F)^{-1}(0)=F^{-1}(F(\Sing F))$ and the same 
multiplicity at each point of $(\kappa\circ F)^{-1}(0)$,  
\item the $C^\omega$ function on $M_m$, defined 
to be $\kappa\circ\tilde F/\kappa\circ F$ on $M_m-(\kappa\circ F)^{-1}(0)$, is close to 1 in the strong Whitney $C^
\infty$ topology. 
Indeed, the map :$C^\omega(M_m)\ni\phi\to\phi\cdot(\kappa\circ F)\in\kappa\circ FC^\omega(M_m)$ is open in the strong 
Whitney $C^\infty$ topology, $\kappa\circ\tilde F$ is contained 
in $\kappa\circ FC^\omega(M_m)$ and close to $\kappa\circ F$ and hence there exist $\phi\in C^\omega(M_m)$ close to 1 
such that $\phi\cdot(\kappa\circ F)=\kappa\circ\tilde F$. 
\end{itemize}
Therefore there exists a $C^\omega$ diffeomorphism $\psi'_m$ of $M_m$ close to id 
in the strong Whitney $C^\infty$ topology such that $\psi'_m((\kappa\circ F)^{-1}(0))=(\kappa\circ F)^{-1}(0)$ and 
$F-\tilde F\circ\psi'_m$ is $r$-flat at $(\kappa\circ F)^{-1}(0)$ for a large integer $r$. 
Then by proposition 4.8,(i) in \cite{FS}, $F$ and $\tilde F$ are $C^\omega$ right equivalent through a $C^\omega$ diffeomorphism 
close to id in the strong Whitney $C^\infty$ topology.

Consider the case of germ on $X$. 
Enlarging $X$ if necessary we assume that $X$ is semialgebraic. 
Set $X_0=X$. 
Let $h_j^{m'},\,g_j^{m'},\,a_{i,j}^{m'},,..,\nu_{i,j}^{m'}$ be the same as above. 
Let $\tilde h_j^0,\,\tilde g_j^0,\,\tilde a_{i,j}^0$ be defined not on $M_0$ but on an open neighborhood $U_0$ 
of $X_0$ in $M_0$ close to $h_j^0,\,g_j^0,\,a_{i,j}^0$, respectively, at $X_0$ in the $C^\infty$ topology so 
that $(\tilde1)$ and $(\tilde2)$ hold on $U_0$. 
Shrink $U_0$ if necessary. 
Then by remark \ref{rm-lem}.(\ref{r2}) of lemma \ref{pert} we have the blowing-up $U_0\times\R^{(k_0+1)^2}\supset U_1
\xrightarrow{\tau_1} U_0$ along center $D_0=Z(\tilde h_j^0)$ defined by $\tilde h_0^0,...,\tilde h
_{k_0}^0$ and analytic embeddings $\psi_0$ of $U_0$ into $M_0$ and $\tilde\psi_0$ of $U_1$ into $M_1$ close to id 
at $X_0$ and at $\tau_1(X)$, respectively, such that $\psi_0(D_0)\subset C_0$ and $\psi_0\circ\tau_1=\pi_1\circ
\tilde\psi_0$.

Next let $\tilde h_j^1,\,\tilde g_j^1,\,\tilde a_{i,j}^1,\,\tilde b_{j_0,j_1}^1,\,\tilde\beta_{i_1,i_2,i_3}^1$ 
be defined on an open neighborhood of $X_0\times X_1$ in $M_0\times\R^{(k_0+1)^2}$ close to $h_j^1,...,\beta_{i_
1,i_2,i_3}^1$, respectively, at $X_0\times X_1$ in the $C^\infty$ topology such that $(\tilde1)$, $(\tilde2)$ 
and $(\tilde 3)$ hold on the neighborhood, where $\tilde\alpha_i^1$ are defined as in the global case and $X_1$ 
denotes a large ball in $\R^{(k_0+1)^2}$ with center 0 such that $\pi_1^{-1}(X_0)$ and $\tau_1^{-1}(X_0)$ are 
contained in $X_0\times\Int X_1$. 
Shrink $U_0$ and the neighborhood of $X_0\times X_1$ so that $U_1$ and $M_1\cap U_0\times\R^{(k_0+1)^2}$ are 
closed subsets of the neighborhood; it is possible because $\pi_1$ and $\tau_1$ are proper. 
Then there exist the blowing-up $U_0\times\R^{(k_0+1)^2}\times\R^{(k_1+1)^2}\supset U_2
\xrightarrow{\tau_2} U_1$ along center $D_1=Z(\tilde h_j^1)\cap U_1$ defined by $\tilde h_0^1,...
,\tilde h_{k_1}^1$ and analytic embeddings $\psi_1$ of $U_1$ into $M_1$ and $\tilde\psi_1$ of $U_2$ into $M_2$ 
close to id at $\tau_1^{-1}(X_0)$ and at $(\tau_1\circ\tau_2)^{-1}(X_0)$, respectively, such that $\psi_1(D_1)
\subset C_1$ and $\psi_1\circ\tau_2=\pi_2\circ\tilde\psi_1$.

Let $1\!<\!m'\!<\!m,\,m''\!<\!m'$ and $m'''\le m'$. 
By induction, let $\tilde h_j^{m'},\,\tilde g_j^{m'},\,\tilde a_{i,j}^{m'},\,\tilde\chi_j^{m'',m'}$,
$\tilde c_{j_0,j_1,j}^{m'',m'}$, $\tilde d_{j,j_1}^{m'',m'},\,\tilde\beta_{i_1,i_2,i_3}^{m'',m'},\,\tilde\beta
_{i_1,i_2,i_3,i_4}^{m''',m'',m'},\tilde\beta_{i_1,i_2,i_3}^{m''',m'',m'}$ be defined on an open neighborhood of 
$X_0\times\cdots\times X_{m'}$ in $M_0\times\R^{(k_0+1)^2}\times\cdots\times\R^{(k_{m'-1}+1)^2}$ close to $h_j^{m'
},\,g_j^{m'},...,$ respectively, at $X_0\times\cdots\times X_{m'}$ in the $C^\infty$ topology such that $(\tilde1),
\,(\tilde2),\,(\tilde4),\,(\tilde5),\,(\tilde6)$ hold on the neighborhood, where $\tilde\alpha_i^{m'}$ are given 
as in the global case and $X_i$ denotes a large ball in $\R^{(k_{i-1}+1)^2}$ with center 0 for $i=2,...,m'$. 
For $m'\,(\le m)\in\N$, let $\tilde\alpha_i^m$ and $\tilde \chi_j^{m',m}$ be 
defined on an open neighborhood of $X_0\times\cdots\times X_m$ close to $\alpha_i^m$ and $\chi_j^{m',m}$, 
respectively, at $X_0\times\cdots\times X_m$ as before, and $\tilde\rho_i,\,\tilde\delta_{i,j},\,\tilde\nu_{i,j}^
{m'}$ on an open neighborhood of $X_0\times\cdots\times X_m\times B$ close to $\rho_i,\,\delta_{i,j},\,\nu_{i,j}
^{m'}$, respectively, at $X_0\times\cdots\times X_m\times B$ so that $(\tilde7)$ and $(\tilde8)$ hold on the 
neighborhood.

Then we obtain a sequence of blowings-up $U_m\xrightarrow{\tau_m} U_{m-1}\longrightarrow\cdots
\xrightarrow{\tau_1} U_0$ along smooth analytic centers $D_{m-1}=Z(\tilde h_j
^{m-1})\cap U_{m-1}$ in $U_{m-1},...,D_0=Z(\tilde h_j^0)$ in $U_0$, respectively, and an 
analytic embedding $\psi:U_m\to M_m$ such that $\psi(\tau_m^{-1}(D_{m-1}))\subset\pi_m^{-1}(C_{m-1}),...,\psi((
\tau_1\circ\cdots\circ\tau_m)^{-1}(D_0))\subset(\pi_1\circ\cdots\circ\pi_m)^{-1}(C_0)$, $f\circ\pi_1\circ\cdots\circ
\pi_m\circ\psi=f\circ\tau_1\circ\cdots\circ\tau_m$, $U_1,...,U_m$ are realized in $U_0\times\PP(k_0)\subset U_0
\times\R^{(k_0+1)^2},...,U_0\times\PP(k_0)\times\cdots\times\PP(k_{m-1})\subset U_0\times\R^{(k_0+1)^2}\times\cdots
\times\R^{(k_{m-1}+1)^2}$, respectively, and each pair $D_i\subset U_i$ and $\psi$ are close to $C_i\subset M_i$ 
at $(\tau_1\circ\cdots\circ\tau_{i-1})^{-1}(X_0)$ and to id at $(\tau_1\circ\cdots\circ\tau_m)^{-1}(X_0)$, 
respectively, in the $C^\infty$ topology. 

Thus it remains only to find the approximations $\tilde h_j^0,\,\tilde g_j^0,...$ of class Nash. This is a consequence of proposition \ref{prop-app} below. 
\end{demo}

\section{Nested Nash approximation}

\subsection{Nash approximation of an analytic diffeomorphism}

In order to prove theorem \ref{thmN} below and theorem \ref{main}, we
need to make a Nash approximation of analytic solutions of
a system of Nash equations. The following proposition is a nested version of the Nash Approximation Theorem established in \cite{CRS}.

\begin{prop}\label{prop-app}
Let $M_1,...,M_m$ be Nash manifolds, $X_1\subset M_1,...,X_m\subset M_m$ compact semialgebraic subsets, and let 
$l_1,...,l_m,n_1,...,n_m\in\N$. 
Let $F_i\in \mathcal N(X_1\times\cdots\times X_i\times\R^{l_1}\times\cdots\times\R^{l_i})^{n_i}$ and $f_i\in\mathcal O(X_1
\times\cdots\times X_i)^{l_i}$ for $i=1,...,m$ such that 
$$F_i(x_1,...,x_i,f_1(x_1),...,f_i(x_1,...,x_i))=0$$ as 
elements of $\mathcal O(X_1\times\cdots\times X_i)^{n_i}$.
Then there exist $\tilde f_i\in \mathcal N(X_1\times\cdots\times X_i)^{l_i}$ close to $f_i$ in the $C^\infty$ topology 
for $i=1,...,m$ such that $F_i(x_1,...,x_i,\tilde f_1(x_1),...,\tilde f_i(x_1,...,x_i))=0$ in $\mathcal N(X_1\times
\cdots\times X_i)^{n_i}$. 
\end{prop}

\begin{demo}
The proof is inspired by Nested Smoothing Theorem 11.4, \cite{Sp} by Teissier and its proof. 
The proof for $m=1$ coincides with Theorem 1.1, \cite{CRS} if $M_1$ is compact and if $X_1=M_1$, and we can prove 
the proposition for $m=1$ in the same way even if $M_1$ is non-compact. 

Regard each $F_i$ as $n_i$ elements of $\mathcal N(X_1\times\cdots\times X_i\times\R^{l_1}\times\cdots\times\R^{l_i})$. 
We can assume that $M_i$ and $X_i$ are all connected and that $F_i$ are polynomial functions in the 
variables $(y_1,...,y_i)\in \R^{l_1}\times\cdots\times\R^{l_i}$ with coefficients in $\mathcal N(X_1\times\cdots\times 
X_i)$ for the same reason as in the proof of Theorem 1.1, \cite{CRS}. 
Let $\mathcal N(X_1\times\cdots\times X_i)[y_1,...,y_i]$ denote the ring of polynomials in the variables $(y_1,...,y_i)
\in\R^{l_1}\times\cdots\times\R^{l_i}$ with coefficients in $\mathcal N(X_1\times\cdots\times X_i)$ and $(F_1,...,F_i)$ 
the ideal of $\mathcal N(X_1\times\cdots\times X_i)[y_1,...,y_i]$ generated by $F_1,...,F_i$.

Consider a commutative diagram of homomorphisms between rings\,:

$$\xymatrix{
\mathcal N(X_1)\ar[r]^{\phi_1}\ar[d]^{p_1} &\frac{\mathcal N(X_1)[y_1]}{(F_1)} \ar[r]^{\psi_1}\ar[d]^{q_1} & \mathcal O(X_1)\ar[d]^{r_1}\\
\mathcal N(X_1\times X_2) \ar[r]^{\phi_2} \ar[d]^{p_2} & \frac{\mathcal N(X_1\times X_2)[y_1,y_2]}{(F_1,F_2)} \ar[r]^{\psi_2} \ar[d]^{q_2} & \mathcal O(X_1\times X_2) \ar[d]^{r_2}\\
\vdots \ar[d]^{p_{m-1}} & \vdots \ar[d]^{q_{m-1}} & \vdots \ar[d]^{r_{m-1}}\\
\mathcal N(X_1\times\cdots\times X_m) \ar[r]^{\phi_m} &\frac{\mathcal N(X_1\times\cdots\times X_m)[y_1,...,y_m]}{(F_1,...,F_m)} \ar[r]^{\psi_m} & \mathcal O(X_1\times\cdots\times X_m),
}$$
where for each $i$, $\phi_i,p_i,q_i$ and $r_i$ are naturally defined, $\psi_i=\id$ on $\mathcal N(X_1\times\cdots\times 
X_i)$ and $\psi_i(y_j)$ is defined to be $f_j$ as an element of $\mathcal O(X_1\times\cdots\times X_i)$ for each $j\le 
i$. 
Then it suffices to find homomorphisms $\tilde\psi_1:\mathcal N(X_1)[y_1]/(F_1)\to \mathcal N(X_1),...,\tilde\psi_m:
\mathcal N(X_1\times\cdots\times X_m)[y_1,...,y_m]/(F_1,...,F_m)\to \mathcal N(X_1\times\cdots\times X_m)$ such that 
$\tilde\psi_1\circ\phi_1=\id,...,\tilde\psi_m\circ\phi_m=\id$, $\tilde\psi_1(y_1),...,\tilde\psi_m(y_m)$ are 
close to $f_1,...,f_m$, respectively, and $r_i\circ\tilde\psi_i=\tilde\psi_{i+1}\circ q_i$ for $0<i<m$. 
For that we only need to decide the values $\tilde\psi_1(y_1),...,\tilde\psi_m(y_m)$ because $\tilde\psi_i(y_k)=
\tilde\psi_k(y_k)$ as elements of $\mathcal N(X_1\times\cdots\times X_i)$ for $k<i$ by the equality $r_j\circ\tilde
\psi_j=\tilde\psi_{j+1}\circ q_j$. 
By \cite{E}, \cite{Fr} and \cite{Ri} we know that $\mathcal O(X_1\times\cdots\times X_i)$ and $N(M_1\times\cdots\times M_i)$ are 
Noetherian, and the proofs in \cite{E} and \cite{Ri} work for the Noetherian property of $\mathcal N(X_1\times\cdots\times X_i)$. 
Hence all the rings in the diagram are Noetherian. 
Therefore, we assume that $\psi_i$ are injective, adding some Nash
functions to $F_i$'s if necessary.

We will find $k_i\in\N$, finite subsets $G_i$ of $\mathcal N(X_1\times\cdots\times X_i)[z_1,...,z_i]$, $z_j\in\R^{k_j
}$, and a commutative diagram of homomorphisms between rings\,:

$$\xymatrix{
\frac{\mathcal N(X_1)[y_1]}{(F_1)} \ar[r]^{\rho_1}\ar[d]^{q_1} &\frac{\mathcal N(X_1)[z_1]}{(G_1)} \ar[r]^{\xi_1}\ar[d]^{s_1}& \mathcal O(X_1)\ar[d]^{r_1}\\
\frac{\mathcal N(X_1\times X_2)[y_1,y_2]}{(F_1,F_2)} \ar[r]^{\rho_2} \ar[d]^{q_2} &\frac{\mathcal N(X_1\times X_2)[z_1,z_2]}{(G_1,G_2)} \ar[r]^{\xi_2}\ar[d]^{s_2} & \mathcal O(X_1\times X_2) \ar[d]^{r_2}\\
\vdots \ar[d]^{q_{m-1}} & \vdots \ar[d]^{s_{m-1}} & \vdots \ar[d]^{r_{m-1}}\\
\frac{\mathcal N(X_1\times\cdots\times X_m)[y_1,...,y_m]}{(F_1,...,F_m)} \ar[r]^{\rho_m} & \frac{\mathcal N(X_1\times\cdots\times X_m)[z_1,...,z_m]}{(G_1,...,G_m)} \ar[r]^{\xi_m} & \mathcal O(X_1\times\cdots\times X_m)
}$$
such that for each $i$, $s_i$ is naturally defined, $\xi_i\circ\rho_i=\psi_i$, $\rho_i=\xi_i=\id$ on $\mathcal N(X_1
\times\cdots\times X_i)$, the zero set $Z_i$ of $(G_1,...,G_i)$ is the germ on $X_1\times\cdots\times X_i\times\R^
{k_1}\times\cdots\times\R^{k_i}$ of a Nash submanifold of $M_1\times\cdots\times M_i\times\R^{k_1}\times\cdots
\times\R^{k_i}$ and $(G_1,...,G_i)$ is the ideal of $\mathcal N(X_1\times\cdots\times X_i)[z_1,...,z_i]$ of function 
germs vanishing on $Z_i$. 
Note that the restriction $\pi_i$ to $Z_i$ of the projection $M_1\times\cdots\times M_i\times\R^{k_1}\times
\cdots\times\R^{k_i}\to M_1\times\cdots\times M_i$ is submersive because $\mathcal N(X_1\times\cdots\times X_i)\subset
\mathcal N(X_1\times\cdots\times X_i)[z_1,...,z_i]/(G_1,...,G_i)$ and $\xi_i|_{\mathcal N(X_1\times\cdots\times X_i)}=\id$, 
that $\xi_i(z_1,...,z_i)$ is an analytic cross-section of $\pi_i$ and that when we regard $M_j$ locally as 
Euclidean spaces the rank of the Jacobian matrix $\frac{D(G_1,...,G_i)}{D(x_1,...,x_i,z_1,...,z_i)}$ equals the 
codimension of $Z_i$ in $M_1\times\cdots\times M_i\times\R^{k_1}\times\cdots\times\R^{k_i}$ at each point of $Z_i$. 

As in the proof of Theorem 1.1, \cite{CRS} there exist $k_i\in\N$, finite subsets $G_i$ of $\mathcal N(X_1\times
\cdots\times X_i)[z_i]$, $z_i\in\R^{k_i}$ and homomorphisms of $\mathcal N(X_1\times\cdots\times X_i)$-algebras 
$$
\frac{\mathcal N(X_1\times\cdots\times X_i)[y_1,...,y_i]}{(F_1,...,F_i)} \xrightarrow{\rho'_i} \frac{\mathcal N(X
_1\times\cdots\times X_i)[z_i]}{(G_i)}\xrightarrow{\xi'_i} \mathcal O(X_1\times\cdots\times X_i)
$$
such that for each $i$, $\xi'_i\circ\rho'_i=\psi_i$, the zero set $Z'_i$ of $(G_i)$ is the germ on $X_1\times\cdots
\times X_i\times\R^{k_i}$ of a Nash submanifold of $M_1\times\cdots\times M_i\times\R^{k_i}$ and $(G_i)$ is the 
ideal of $\mathcal N(X_1\times\cdots\times X_i)[z_i]$ of function germs vanishing on $Z'_i$. 
Then as above the restriction $\pi'_i$ to $Z'_i$ of the projection $M_1\times\cdots\times M_i\times\R^{k_i}\to M_1
\times\cdots\times M_i$ is submersive and $\xi'_i(z_i)$ is an analytic cross-section of $\pi'_i$. 
Define $\rho_i$ to be the composition of $\rho'_i$ with the canonical homomorphism $\mathcal N(X_1\times\cdots\times X_
i)[z_i]/(G_i)\to\mathcal N(X_1\times\cdots\times X_i)[z_1,...,z_i]/(G_1,...,G_i)$ and $\xi_i$ by $\xi_i(z_i)=\xi'_i(z_i
)$ and $\xi_i(z_j)=r_{i-1}\circ\cdots\circ r_j\circ\xi'_j(z_j)$. 
Then the conditions on $G_i,\rho_i$ and $\xi_i$ are satisfied. 
Indeed, first the zero set $Z_i$ of $(G_1,...,G_i)$ in $M_1\times\cdots\times M_i\times\R^{k_1}\times\cdots\times\R
^{k_i}$ is the fiber product of the submersions $(\pi'_1,\id):Z'_1\times M_2\times\cdots\times M_i\to M_1\times
\cdots\times M_i$, $(\pi'_2,\id):Z'_2\times M_3\times\cdots\times M_i\to M_1\times\cdots\times M_i,...,\pi'_i:Z'_i
\to M_1\times\cdots\times M_i$ and hence the germ on $X_1\times\cdots\times X_i\times\R^{k_1}\times\cdots\times\R^
{k_i}$ of some Nash submanifold of $M_1\times\cdots\times M_i\times\R^{k_1}\times\cdots\times\R^{k_i}$. 
Next, add some finite subset of $\mathcal N(X_1\times\cdots\times X_i)[z_1,...,z_i]$ to $G_i$ whose elements vanish on 
$Z_i$, if necessary. 
Then $(G_1,...,G_i)$ is the ideal of $\mathcal N(X_1\times\cdots\times X_i)[z_1,...,z_i]$ of function germs vanishing 
on $Z_i$. 
Thus we obtain the required diagram.

For the construction of $\tilde\psi_i$'s it suffices to find homomorphisms of $\mathcal N(X_1\times\cdots\times X_i)
$-algebras $\tilde\xi_i:\mathcal N(X_1\times\cdots\times X_i)[z_1,...,z_i]/(G_1,...,G_i)\to\mathcal N(X_1\times\cdots\times 
X_i)$ so that $\tilde\xi_i(z_i)\in\mathcal N(X_1\times\cdots\times X_i)^{k_i}$ are close to $\xi_i(z_i)\in\mathcal O(X_1
\times\cdots\times X_i)^{k_i}$ in the $C^\infty$ topology because if we define $\tilde\xi_i$ by $\tilde\xi_i(z_j)=
r_{i-1}\circ\cdots r_j\circ\tilde\xi_j(z_j)$ for $j<i$ then $\tilde\psi_i=\tilde\xi_i\circ\rho_i$ fulfill the 
requirements. 
By induction on $m$ we assume that $\tilde\xi_1,...,\tilde\xi_{m-1}$ are given. 
Then as before we only need to decide $\tilde\xi_m(z_m)\in\mathcal N(X_1\times\cdots\times X_m)^{k_m}$ close to $\xi_m
(z_m)\in\mathcal O(X_1\times\cdots\times X_m)^{k_m}$ in the $C^\infty$ topology so that $G_m(x_1,...,x_m,\tilde\xi_m(z
_m))=\{0\}$ as a subset of $\mathcal N(X_1\times\cdots\times X_m)$, i.e. $\tilde\xi_m(z_m)$ is a Nash cross-section 
of $\pi'_m$. 
(Here the elements of $G_m$ may be of the variables $x_1,...,x_m,z_1,...,z_m$. 
However, we can remove some elements from $G_m$ so that they are all in the variables $x_1,...,x_m,z_m$ by the 
above arguments.) 
Let $U\subset U'$ be small open semialgebraic neighborhoods of $X_1\times\cdots\times X_m$ in $M_1\times\cdots
\times M_m$ such that $\overline U$ is compact and contained in $U'$, $Z'_m$ is the germ on $X_1\times\cdots
\times X_m\times\R^{k_m}$ of a closed Nash submanifold $Z'$ of $U'\times\R^{k_m}$, $\pi'_m$ is the germ on $X_1
\times\cdots\times X_m\times\R^{k_m}$ of a surjective submersion $\pi':Z'\to U'$ and $\xi_m(z_m)$ is the germ on 
$X_1\times\cdots\times X_m$ of an analytic cross-section $\xi:\overline U\to Z'$ of $\pi'$. 
Let $\eta$ be a Nash approximation of $\xi|_{\overline U}:\overline U\to Z'$ in the $C^\infty$ topology (Nash 
Approximation Theorem), which is an embedding but not necessarily a cross-section of $\pi'|_{\pi^{\prime-1}(
\overline U)}$. 
Let $\tilde\xi_m(z_m)$ be the germ of $\eta\circ(\pi'_m\circ\eta)^{-1}$ on $X_1\times\cdots\times X_m$. 
Then $\tilde\xi_m(z_m)$ is a Nash cross-section of $\pi'_m$ close to $\xi_m(z_m)$ in the $C^\infty$ topology. 
Thus we complete the proof. 
\end{demo}

As a corollary of proposition \ref{prop-app} we obtain the following Nash approximation theorem, which generalizes that proved in \cite{CRS} in the case where $X=M$ and $M$ is compact. 
\begin{thm}\label{thmN}
Let $M$ be a Nash manifold, $X\subset M$ be a compact semialgebraic subset, and $f,g$ be Nash function germs on 
$X$ in $M$. 
If $f$ and $g$ are analytically right equivalent, then $f$ and $g$ are Nash right equivalent. 
The diffeomorphism of Nash right equivalence can be chosen to be close to the given one of analytic right 
equivalence in the $C^\infty$ topology. 
\end{thm}

Here we naturally define analytic or Nash right equivalence of two analytic or Nash function germs, respectively, 
on $X$ in $M$. 
We note only that the diffeomorphism germ of equivalence is $X$-preserving.

For the proof we introduce some notions. 
Let $X$ be a semialgebraic subset of a Nash manifold $M$. 
We consider the germs of sets on $X$ in $M$. 
For a germ $A$ on $X$ of a subset of $M$, let $\overline{A}^X$ or $A^{-X}$ denote 
the {\it Nash closure} of $A$ in $M$, i.e. the smallest Nash set germ in $M$ containing $A$. 
In the case where $A$ is a subset of $M$ also, $\overline{A}^X$
coincides with the Nash closure of the 
germ of $A$ on $X$ in $M$. 
We define by induction a sequence of Nash set germs $M_i$ in $M$ as follows. 
Let $M_1$ be the germ $\overline X^X$ and assume that $M_1,...,M_{k-1}$ are given for $k\,(>1)\in\N$. 
Then, set 
$$
M_k=[\overline{(M_{k-1}-X)}\cap\overline{(M_{k-1}\cap X)}\,\overline]^X. 
$$
We call $\{M_i\}$ the {\it canonical Nash germ decomposition} of $X$. 
Then $\{M_i\}$ is a decreasing sequence of Nash set germs, for each $i$ the set $X\cap M_i-M_{i+1}$ is a union of some 
connected components of $M_i-M_{i+1}$ and $\{M_i\}$ is canonical in the following sense. 
Let $\{M'_i\}$ be another decreasing sequence of Nash set germs such that for each $i$ the set 
$X\cap M'_i-M'_{i+1}$ is a union of some connected components of $M'_i-M'_{i+1}$, which is called a {\it Nash 
germ decomposition of $X$}. 
Assume that $\{M'_i\}$ is distinct from $\{M_i\}$. 
Then $M'_1=M_1,...,M'_{k-1}=M_{k-1}$ and $M'_k\supsetneq M_k$ for some $k$. 

A subset $Y$ of an analytic manifold $N$ is called {\it global semianalytic} if $Y$ is described by finitely 
many equalities and inequalities of global analytic functions on $N$. 
Let $Y$ be a relatively compact and global semianalytic subset of $N$. 
Then we can define the {\it global analytic closure} $\overline Z^Y$ of the germ on $Y$ of a subset $Z$ of $N$ 
(or of the germ $Z$ on $Y$ of a subset of $N$) and a (the {\it canonical) global analytic germ decomposition} 
of $Y$ in the same way. Indeed, for a global semianalytic set $Z$ in $N$, $\dim Z=\dim\overline Z^Y$ for 
the reason explained below, and if $Z$ is, moreover, relatively compact then $\overline Z$ is global 
semianalytic by \cite{Ru} and finally, a global analytic set is global semianalytic (and moreover is the zero set 
of one global analytic function). 
To prove that $\dim Z=\dim\overline Z^Y$ we can assume that $Z$ is a global semianalytic set of the form $\{x\in 
N:f(x)=0,f_1>0,...,f_k(x)>0\}$ for some analytic functions $f,f_1,...,f_k$ on $N$ dividing $Z$ if necessary, and 
it suffices to prove that the global analytic closure $\overline Z^N$ of $Z$ is of the same dimension as $Z$. 
Let $x_0\in Z$ where the germ of $\overline Z^N$ is of dimension
$\dim\overline Z^N$. There exists such a point since $Z\cap
\Reg\overline Z^N\not=\emptyset$. 
Then $f_1>0,...,f_k>0$ on a neighborhood of $x_0$ in $\overline Z^N$. 
Hence $Z$ contains the neighborhood and is of dimension $\dim\overline Z^N$. 
(We do not know whether the canonical global analytic germ decomposition of $Y$ exists if $Y$ is a 
non-relatively compact global semianalytic set.) 
\begin{rmk}\label{rmk3}\begin{flushleft}\end{flushleft}
\begin{enumerate}
\item \label{i} Let $N\supset Y$ and $N'\supset Y'$ be analytic manifolds and respective relatively compact and global 
semianalytic subsets and $\phi:N\to N'$ an analytic diffeomorphism such that $\phi(Y)=Y'$. 
Then $\phi$ carries the canonical global analytic germ decomposition of $Y$ to the canonical global analytic germ 
decomposition of $Y'$. 
\item \label{ii} Let $M\supset X$ be a Nash manifold and a semialgebraic subset. 
Then the canonical global analytic germ decomposition of $X$ is well-defined and coincides with the canonical Nash 
germ decomposition of $X$ because the global analytic closure of a semialgebraic set equals its Nash closure. 
\end{enumerate}
\end{rmk}

{\it Proof of theorem \ref{thmN}.}
Let $M\subset\R^n$, set $M_0=\overline M^X$, and let $\{M_i:i=1,2,...\}$ be the canonical Nash germ decomposition 
of $X$. 
Let $\mathcal O(X)$ and $\mathcal N(X)$ denote respectively the germs of analytic and Nash functions on $X$ in $\R^n$ 
but not in $M$. 
Let $\{\phi_{i,j}:j\}$ for each $i=0,1,...$ be finitely many generators of the ideal of $\mathcal N(X)$ defined by 
$M_i$. 
Extend $f$ and $g$ to elements $\hat f$ and $\hat g$ of $\mathcal N(X)$, respectively. 
Then we have $\pi=(\pi_1,...,\pi_n)\in\mathcal O(X)^n$ such that $\pi|_M$ is the germ on $X$ of a $C^\omega$ 
diffeomorphism between neighborhoods of $X$ in $M$ and $f\circ\pi=g$ on $M$. 
Hence there exist $\alpha_j\in\mathcal O(X)$ such that 
\begin{equation}\tag{1}\label{*}
\hat f\circ\pi=\hat g+\sum_j\alpha_j\phi_{0,j}.
\end{equation}

By remarks \ref{rmk3}.(\ref{i}) and \ref{rmk3}.(\ref{ii}), $\pi$ is $M_i$-preserving. 
Hence there exist $\beta_{i,j,j'}\in\mathcal O(X)$ such that for each $\phi_{i,j}$ 

\begin{equation}\tag{2}\label{**}
\phi_{i,j}\circ\pi=\sum_{j'}\beta_{i,j,j'}\phi_{i,j'}.
\end{equation}

Apply proposition \ref{prop-app} to $(\ref{*})$ and $(\ref{**})$. 
Then there exist $\tilde\pi\in\mathcal N(X)^n$, $\tilde\alpha_j\in\mathcal N(X)$ and $\tilde\beta_{i,j,j'}\in\mathcal N(X)$ 
close to $\pi,\alpha_j$ and $\beta_{i,j,j'}$, respectively, in the $C^\infty$ topology such that

\begin{equation}\tag{$\widetilde 1$}\label{11}
\hat f\circ\tilde\pi=\hat g+\sum_j\tilde \alpha_j\phi_{0,j},
\end{equation}

\begin{equation}\tag{$\widetilde 2$}\label{22}
\phi_{i,j}\circ\tilde\pi=\sum_{j'}\tilde\beta_{i,j,j'}\phi_{i,j'}. 
\end{equation}

Since $\tilde\pi$ is an approximation of $\pi$, (\ref{22}) implies that $\tilde \pi|_M$ is the germ on $X$ 
of a Nash diffeomorphism between open semialgebraic neighborhoods of $X$ in $M$. 
Hence by (\ref{11}), $f\circ(\tilde\pi|_M)=g$, and the theorem is proved. 
\begin{flushright}
$\Box$
\end{flushright}

Consider the plural case of $\{X\}$. 
Let $X$ and $X_j,\ j=1,...,k$, be semialgebraic subsets of a Nash manifold $M$. 
We define the {\it canonical Nash germ decomposition} $\{M_i\}$ of $\{X;X_j\}$ as follows. 
Set $X_0=\cup_{j=1}^k X_j$ and $M_1=\overline X_0^X$. Assume that we
have defined $M_j$ for $j \leq i$. 
Then we set 
$$
M_{i+1}=(\cup_{j=1}^k[\overline{(M_i-X_j)}\cap\overline{(M_i\cap X_j)}]\,\overline)^X. 
$$
The same properties as in the single case hold. 
To be precise, $\{M_i\}$ is a decreasing sequence of Nash set germs on $X$, for each $i$ and $j>0$ $X_j\cap M_i-M
_{i+1}$ is a union of some connected components of $M_i-M_{i+1}$, and $\{M_i\}$ is canonical in the same sense as 
in the single case. 
We define also a {\it Nash germ decomposition} of $\{X;X_j\}$ and a (the {\it canonical) global analytic germ 
decomposition} of a finite family of relatively compact global semianalytic sets in an analytic manifold in the 
same way. 
Then remark \ref{rmk3}.(\ref{i}) and \ref{rmk3}.(\ref{ii}) hold also in the plural case.

Using these notions and remarks in the same way as above we can refine theorem \ref{thmN} as follows. 
\begin{rmk}\label{iii} In theorem \ref{thmN}, let $A_i$ and $B_i$ be a finite number of semialgebraic subsets of $M$ such that the 
diffeomorphism germ of analytical right equivalence carries the germ on $X$ of each $A_i$ to the one of $B_i$. 
Then the diffeomorphism germ of Nash right equivalence is chosen to have the same property.

In particular, if we set $f=g=$\,constant then we have the following statement.

Let $M$ and $X$ be the same as in theorem \ref{thmN}, and let $C_i$ and $D_i$ be finitely many semialgebraic subsets of 
$M$. 
Assume that there exists a germ $\pi$ on $X$ of an analytic diffeomorphism between neighborhoods of $X$ in $M$ 
which carries the germ on $X$ of each $C_i$ to the one of $D_i$ and such that $\pi(X)=X$. 
Then $\pi$ is approximated by a germ on $X$ of a Nash diffeomorphism between semialgebraic neighborhoods of $X$ 
in $M$ in the $C^\infty$ topology keeping the properties of $\pi$. 
\end{rmk}

\subsection{Proof of theorem \ref{main}}

\subsubsection{Proof of theorem \ref{main} in the case where $X=M$ and $M$ is compact}
Assume that $f$ and $g$ are almost Blow-analytically equivalent. 
Let $\pi_f:N\to M,\,\pi_g:L\to M$ and $h:N\to L$ be two compositions of finite sequences of blowings-up along smooth 
analytic centers and an analytic diffeomorphism, respectively, such that $f\circ\pi_f=g\circ\pi_g\circ h$. 
We can assume that $f\circ\pi_f$ and $g\circ \pi_g$ have only normal crossing singularities, performing if
necessary additional blowings-up.

Then by theorem \ref{Nresol} there exist compositions of finite sequences of blowings-up along smooth Nash centers $\tilde\pi
_f:\tilde N\to M$ and $\tilde\pi_g:\tilde L\to M$ and analytic diffeomorphisms $\tau_f:\tilde N\to N$ and $\tau_g:
\tilde L\to L$ such that $f\circ\pi_f\circ\tau_f=f\circ\tilde\pi_f$ and $g\circ\pi_g\circ\tau_g=g\circ\tilde\tau_g$. 
Hence $f\circ\tilde\pi_f\circ\tau_f^{-1}=g\circ\tilde\pi_g\circ\tau_g^{-1}\circ h$, and $f\circ\tilde\pi_f$ and $g
\circ\tilde\pi$ are analytically right equivalent. 
Then by theorem \ref{thmN} they are Nash right equivalent, i.e. $f$ and $g$ are almost Blow-Nash equivalent.

Next we prove that if $f$ and $g$ are almost Blow-analytically R-L equivalent then they are almost Blow-Nash R-L 
equivalent. 
For that it suffices to prove that two analytically R-L equivalent Nash functions $\phi$ and $\psi$ with only normal 
crossing singularities are Nash R-L equivalent. 
Let $\pi$ and $\tau$ be analytic diffeomorphisms of $M$ and $\R$, respectively, such that $\tau\circ\phi=\psi\circ\pi$. 
Then $\pi(\Sing\phi)=\Sing \psi$, $\tau(\phi(\Sing\phi))=\psi(\Sing\psi)$ and $\pi(\phi^{-1}(\phi(\Sing\phi)))=\psi^
{-1}(\psi(\Sing\psi))$. 
By remark \ref{iii} we have a Nash diffeomorphism $\pi_0$ of $M$ close to $\pi$ in the $C^\infty$ topology such that $\pi
_0(\Sing\phi)=\Sing\psi$ and $\pi_0(\phi^{-1}(\phi(\Sing\phi)))=\psi^{-1}(\psi(\Sing\psi))$, and since $\phi(\Sing
\phi)$ is a finite set, we have a Nash diffeomorphism $\tau_0$ of $\R$ close to $\tau$ in the compact-open $C^\infty$ 
topology such that $\tau_0=\tau$ on $\phi(\Sing\phi)$. 
Replace $\psi$ with $\tau_0^{-1}\circ\psi\circ\pi_0$. 
Then we can assume from the beginning that $\Sing\phi=\Sing\psi$, $\phi(\Sing\phi)=\psi(\Sing\psi)$, $\phi^{-1}(\phi
(\Sing\phi))=\psi^{-1}(\psi(\Sing\psi))$, and $\pi$ and $\tau$ are close to id in the $C^\infty$ topology and in the 
compact-open $C^\infty$ topology, respectively. 
Hence for each $z_0\in\phi(\Sing\phi)$, $\phi-z_0$ and $\psi-z_0$ have the same sign at each point of $M$ and the 
same multiplicity at each point of $\phi^{-1}(z_0)$. 
Let $\rho$ be a Nash function on $\R$ with zero set $\phi(\Sing\phi)$ and regular there. 
Then $\rho\circ\phi$ and $\rho\circ\psi$ satisfy the conditions in lemma 4.7, \cite{FS}---$(\rho\circ\phi)^{-1}(0)=(\rho\circ
\psi)^{-1}(0)\,(=\phi^{-1}(\phi(\Sing\phi)))$, $\rho\circ\phi$ and $\rho\circ\psi$ have the same sign at each point 
of $M$, only normal crossing singularities at $(\rho\circ\phi)^{-1}(0)$ and the same multiplicity at each point of 
$(\rho\circ\phi)^{-1}(0)$, and the natural extension to $M$ of the function $\rho\circ\psi/\rho\circ\phi$ defined 
on $M-(\rho\circ\phi)^{-1}(0)$ is close to 1 in the $C^\infty$ topology. 
Hence by lemma 4.7 in \cite{FS} there exists a Nash diffeomorphism $\pi_1$ of $M$ close to id in the $C^\infty$ topology such 
that $\pi_1(\phi^{-1}(\phi(\Sing\phi)))=\phi^{-1}(\phi(\Sing\phi))$ and $\phi-\psi\circ\pi_1$ is $l$-flat at $\phi^{
-1}(\phi(\Sing\phi))$ for a large integer $l$. 
Replace, once more, $\psi$ with $\psi\circ\pi_1$. 
Then we can assume, moreover, that $\phi-\psi$ is $l$-flat at $\phi^{-1}(\phi(\Sing\phi))$ and close to 0 in the 
$C^\infty$ topology. 
Hence by proposition 4.8,(i) in \cite{FS}, $\phi$ and $\psi$ are analytically right equivalent and then by theorem \ref{thmN} they are 
Nash right equivalent.

\subsubsection{Proof of theorem \ref{main} in the case $X\subset\Sing f$}
Assume that $f$ and $g$ are Nash functions defined on open semialgebraic neighborhoods $U$ and $V$, respectively, 
of $X$ in $M$, and let $\pi_f:N\to U$, $\pi_g:L\to V$ and $h:N'\to L'$ be two compositions of finite sequences of 
blowings-up along smooth analytic centers and an analytic diffeomorphism from an open neighborhood $N'$ of $\pi^{-1}
_f(X)$ in $N$ to one $L'$ of $\pi^{-1}_g(X)$ in $L$, respectively, such that $f\circ\pi_f=g\circ\pi_g\circ h$ and 
$h(\pi^{-1}_f(X))=\pi^{-1}_g(X)$. 
When we proceed as in the proof in the case of $X=M$ we can replace $\pi_f:N\to U,\,\pi_g:L\to V$ and $h:N'\to L'$ 
by Nash $\tilde\pi_f:\tilde N\to\tilde U,\,\tilde\pi_g:\tilde L\to\tilde V$ and $h:\tilde N'\to\tilde L'$, 
respectively, so that $f\circ\tilde\pi_f=g\circ\tilde\pi_g\circ\tilde h$. 
However, we cannot expect the equality $\tilde h(\tilde\pi_f^{-1}(X))=\tilde\pi_g^{-1}(X)$. 
For the equality we need to modify $\pi_f$ and $\pi_g$.

As in the construction of the canonical Nash germ decomposition we have a decreasing sequence of Nash sets $X_i,\,
i=1,2,...,$ in $U$ such that $X_1$ is the Nash closure of $X$ in $U$ and for each $i$ the set $X_i\cap X-X_{i+1}$ is a union 
of some connected components of $X_i-X_{i+1}$. 
Set $X_{f,i}=\pi^{-1}_f(X_i)$. 
Then $\{X_{f,i}\}$ is a decreasing sequence of global analytic sets in $N$, $\pi^{-1}_f(X)\subset X_{f,1}$, and for 
each $i$ the set $\pi_f^{-1}(X)\cap X_{f,i}-X_{f,i+1}$ is a union of
some connected components of $X_{f,i}-X_{f,i+1}$. 

Now, by Hironaka Desingularization
Theorem (and similarly to the case of $X=M$), we are able to reduce the problem to the case 
where $X_{f,i}$ are normal crossing, $f\circ\pi_f$ has only normal crossing singularities, and hence $\pi_f^{-1}(X)$ 
is a union of some connected components of strata of the canonical stratification of $\Sing(f\circ\pi_f)$. We call these properties $(*)$. 
Shrinking $N',\,V,\,L$ and $L'$ so that $L'=L$ if necessary, then
$\pi_g^{-1}(X)$ and $g\circ\pi_g$ satisfy also $(*)$.

Let $N$ and $L$ be realized in $U\times\PP(k_{f,0})\times\cdots\times\PP(k_{f,m'-1})$ and in $V\times\PP(k_{g,0})
\times\cdots\times\PP(k_{g,m''-1})$, respectively, as in theorem \ref{Nresol}. 
Then by theorem \ref{Nresol} there exist compositions of finite sequences of blowings-up along smooth Nash centers $\tilde
\pi_f:\tilde N\to\tilde U$ and $\tilde\pi_g:\tilde L\to\tilde V$ and analytic embeddings $h_f:\tilde N\to N'$ and 
$h_g:\tilde L\to L'$ such that
\begin{itemize}
\item $\tilde U$ and $\tilde V$ are open semialgebraic neighborhoods of $X$ in $U$ and $V$, 
respectively,
\item  $f\circ\pi_f\circ h_f=f\circ\tilde\pi_f,\,g\circ\pi_g\circ h_g=g\circ\tilde\pi_g$, $\tilde N$ and 
$\tilde L$ are realized in $\tilde U\times\PP(k_{f,0})\times\cdots\times\PP(k_{f,m'-1})$ and in $\tilde V\times\PP(k_
{g,0})\times\cdots\times\PP(k_{g,m''-1})$, respectively, 
\item $\tilde N$ and $\tilde L$ are close to $N$ and $L$ at 
$\tilde\pi_f^{-1}(X)$ and $\tilde\pi_g^{-1}(X)$, respectively, in the $C^\infty$ topology, 
\item $(**)$ $h_f$ and 
$h_g$ are close to id at $\tilde\pi_f^{-1}(X)$ and $\tilde\pi_g^{-1}(X)$, respectively, in the $C^\infty$ topology. 
\end{itemize}
Hence 
$$
f\circ\tilde\pi_f\circ h_f^{-1}=g\circ\tilde\pi_g\circ h^{-1}_g\circ h\quad\text{on}\ \Ima h_f\cap h^{-1}(\Ima h_g). 
$$
Clearly $h_f(\Sing(f\circ\tilde\pi_f))\subset\Sing(f\circ\pi_f)$ and $h_g(\Sing(g\circ\tilde\pi_g))\subset\Sing(g
\circ\pi_g)$. 
It follows from $(*)$ and $(**)$ that $\tilde\pi_f^{-1}(X)$ and $\tilde\pi_g^{-1}(X)$ are unions of some connected 
components of strata of the canonical stratifications of $\Sing(f\circ\tilde\pi_f)$ and $\Sing(g\circ\tilde\pi_g)$, 
respectively, and hence 
$$
h_f(\tilde\pi_f^{-1}(X))=\pi_f^{-1}(X)\quad\text{and}\quad h_g(\tilde\pi_g^{-1}(X))=\pi_g^{-1}(X). 
$$
Therefore, the germs of $f\circ\tilde\pi_f$ on $\tilde\pi_f^{-1}(X)$ and of $g\circ\tilde\pi_g$ on $\tilde\pi_g^{-1}
(X)$ are analytically right equivalent. 
On the other hand, by remark \ref{iii}, the germs of $\tilde N$ on $\tilde\pi_f^{-1}(X)$ and of $\tilde L$ on 
$\tilde\pi_g^{-1}(X)$ are Nash diffeomorphic. 
Hence we can regard them as the same Nash set germ. 
Then by theorem \ref{thmN} and remark \ref{iii}, the germs of $f\circ\tilde\pi_f$ on $\tilde\pi_f^{-1}(X)$ and of $g\circ
\tilde\pi_g$ on $\tilde\pi_g^{-1}(X)$ are Nash right equivalent. 
Thus the germs of $f$ and $g$ on $X$ are almost Blow-Nash equivalent.

Finally, the case of the R-L equivalences runs in the same way as that of $X=M$.

\end{document}